\documentclass[onecolumn,draftclsnofoot]{IEEEtran}

\usepackage{amsthm,amsfonts,amsmath,amssymb,dsfont,color, multirow}
\usepackage{commath,enumerate, mathtools}
\newtheorem{thm}{\textbf{Theorem}}[]
\newtheorem*{thm*}{\textbf{Theorem}}
\newtheorem{coro}{\textbf{Corollary}}[]
\newtheorem*{coro*}{\textbf{Corollary}}
\newtheorem{lem}{\textbf{Lemma}}

\newtheorem{rem}{\textbf{Remark}}
\newtheorem{defn}{\textbf{Definition}}
\usepackage{subcaption, multirow} 
\newcommand{\mvec}[1]{{\boldsymbol #1}}

\def\F{\mathcal{F}}

\def\rme{{\rm e}}
\newcommand{\p}[2][{}]{\mathbb{P}_{#1} \intoo{#2}}
\newcommand{\e}[2][{}]{\mathbb{E}_{#1} \sbr{#2}}
\newcommand{\expp}[1]{\exp \intoo{#1}}
\newcommand{\define}{\triangleq}
\definecolor{Milad}{RGB}{200, 0, 100}
\definecolor{Reviewer1}{RGB}{42, 0, 208}
\definecolor{Reviewer2}{RGB}{192, 9, 9}
\definecolor{Reviewer3}{RGB}{17, 117, 12}

\newcommand{\milad}[1]{  #1 }

\newcommand{\rii }[1]{  #1 }
\newcommand{\riii }[1]{  #1 }

\begin{document}

\title{Sharp Concentration Results for Heavy-Tailed Distributions}

%

\author{Milad Bakhshizadeh, Arian Maleki, Victor H. de la Pena}
\maketitle

\begin{abstract}
We obtain concentration and large deviation for the sums of independent and identically distributed random variables with heavy-tailed distributions. Our concentration results are concerned with random variables whose distributions satisfy $\p{X>t} \leq {\rm e}^{- I(t)}$, where $I: \mathds{R} \rightarrow \mathds{R}$ is an increasing function and $I(t)/t \rightarrow \alpha \in \intco{0, \infty}$ as $t \rightarrow \infty$. Our main theorem can not only recover some of the existing results, such as the concentration of the sum of subWeibull random variables, but it can also produce new results for the sum of random variables with heavier tails. We show that the concentration inequalities we obtain are sharp enough to offer large deviation results for the sums of independent random variables as well. Our analyses which are based on standard truncation arguments simplify, unify and generalize the existing results on the concentration and large deviation of heavy-tailed random variables. 

\\
2000 Math Subject Classification: 60B10, 60F05, 60F10
\end{abstract}

\section{Introduction} \label{sec:introduction}

The concentration of measure inequalities have recently received substantial attention in high-dimensional statistics and machine learning \cite{vershynin_non-asymptotic}. While concentration inequalities are well-understood for subGaussian and subexponential random variables, in many application areas, such as signal processing  \cite{coper}, machine learning \cite{concentration_for_ML} and  optimization \cite{gurbuzbalaban2020heavy, zhang2019adaptive,  gorbunov2020stochastic}
 we need concentration results for sums of random variables with heavier tails. For instance, \cite{concentration_for_ML}  shows how the class of subWeibull random variables naturally appear in the context of neural networks.
 The standard technique, i.e. finding upper bounds for the moment generating function (MGF), clearly fails for heavy-tailed distributions whose moment generating functions do not exist. Furthermore, other techniques, such as Chebyshev's inequality, are incapable of obtaining sharp results. The goal of this paper is to show that under quite general conditions on the tail, a simple truncation argument can not only help us use the standard MGF argument for heavy-tailed random variables, but is also capable of obtaining sharp concentration results.

Another common technique to obtain concentration inequalities is to consider an Orlicz norm.  Orlicz norm of a random variable $X$ is defined as

\begin{equation} \label{eq:orlicz norm}
    \norm{X}_{\psi} \define \inf \cbr{\lambda > 0 : \e{\psi \intoo{\frac{\abs{X}}{\lambda}}} \leq 1},
\end{equation}
for some convex positive function $\psi: \mathds{R}^+ \to \mathds{R}^+$.  Having \eqref{eq:orlicz norm} one can obtain
\begin{equation} \label{eq:concentration from orlicz norm}
    \p{\sum_{i = 1}^n X_i > t} =  \p{\psi \intoo{ \frac{1}{\norm{\sum X_i}_{\psi}}\sum_{i = 1}^n X_i }> \psi \intoo{ \frac{t}{\norm{\sum X_i}_{\psi}}}} \leq {\psi \intoo{\frac{t}{\norm{\sum X_i}_{\psi}}}}^{-1}.
\end{equation}
Several previous works proposed suitable functions $\psi$ for obtaining concentration inequalities for specific classes of distributions, such as subWeibull distributions \cite{ledoux2013probability,  adamczak2008tail, arun18beyond-sub-gaussian, gorbunov2020stochastic, chamakh2021orlicz}. While Orclisz norm techniques are powerful in obtaining concentration results, they suffer from the following limitations:
\begin{enumerate}
    \item Because of the complex algebraic form of the Orclisz norm, the bounds that are obtained by this approach are very rigid. Hence, optimizing constants to obtain more accurate bounds are usually not possible. We will clarify this point in Section \ref{sec:compall}. Furthermore, the function $\psi$ is tailored to the distribution of $X_i$, and novel Orclisz norms shall be created (and their sharpness measured) for any new distribution. We will show how our approach resolves both issues in the next section.  
    
    \item It is usually hard to have an intuitive explanation of the behavior of a random sum through its Orlicz norm.  This is in contrast with truncation technique that not only does upper bound the probability of deviation, but also reveals the event which is responsible for large deviations of the sum.  This claim will become clearer in Section \ref{sec:main results}.

\end{enumerate}
 The truncation technique that will be pursued in this paper, not only deals with simpler terms in the final result, but also relates the tail bound of the sum to the tail of one summand which is an intuitive and interpretable quantity.  In fact, it is clear from the tail bound we obtain in this paper that for heavy tailed distributions deviation of only one summand is responsible for the large deviation of the sum even when the number of summands grows to infinity.  Another advantage of our approach to the Orlicz-norm is its generality. We use the same technique and obtain sharp concentration for all distributions with bounded second moments. 

Other researchers have also studied the problem of obtaining concentration results for sums of heavy-tailed random variables  \cite{nagaev1979large, klass97, klass07, klass16, hitczenko, hitczenko97}.  For instance, \cite{nagaev1979large} discusses several inequalities for finite sums of independent random variables with variety of tail decays.  The proof techniques of the present paper have a similar flavor to what is used in \cite{nagaev1979large}; we also use the truncation of random variables and bound the MGF of the truncated random variables. The generality of the inequalities presented in \cite{nagaev1979large} in terms of the truncation levels, the moments of random variables, etc., makes the results difficult to use and interpret. In particular, obtaining the optimal choice of the parameters that appear in different upper bounds and simplifying the expressions for a given set of parameters is a time-consuming and cumbersome task. We will clarify of all these points with more detailes in Section \ref{sec:compall}. Compared to \cite{nagaev1979large}, we only consider the sum of random variables with bounded variances. For this class of distributions we are able to find the optimal truncation level (and optimal choice of other parameters involved in our approach). Using this right truncation level, we have been able to reduce the problem of obtaining sharp concentration results to that of finding an upper bound for the expectation of a smooth function of an individual random variable. We have also offered several insights on the quantity that is involved in our upper bound. As a result, our concentration results, while less general than \cite{nagaev1979large}, are much more interpretable and the calculations that are involved in them can be easily carried out. Despite the simpler form of our results, as we show through large deviation, they are still sharp. We should also emphasize that there have been more follow-up researches \cite{klass97, klass07, klass16, hitczenko, hitczenko97, FAN20123545} that have appeared after \cite{nagaev1979large}. These works suffer from similar issues as the ones we discussed about \cite{nagaev1979large}. For instance, the bounds in \cite{klass97} are written in terms of the solutions of some optimization problems which are not easily solvable for most distributions of interest.

 To explore the accuracy of our concentration approach, we use our technique to obtain large deviation results. Not surprisingly, the tools we offer for our concentration results are also able to obtain the large deviation results that are consistent with the existing literature on the large deviation behavior of sums of independent, heavy-tailed random variables \cite{rozovskii1994probabilities,  rozovskii1990probabilities, denisov2008large, kontoyiannis2006measure, bazhba2017sample, borovkov2000large, borovkov2006integro}.

\section{Our main contributions}
\label{sec:main results}

\subsection{Concentration}\label{ssec:concentration}

\riii{We develop a ready-to-use concentration inequality for the sum of independent and identically distributed heavy-tailed random variables. This inequality reduces the problem of finding a concentration inequality to the problem of upper bounding the expectation of a smooth (and relatively simple) function of a random variable. We also provide simple-to-use upper bounds for this term. We show that our concentration result is sufficiently sharp and that it is general enough to cover all distributions with finite second moments. }
 Let us start with the following definition. 

\begin{defn} \label{def:tail capture}
Let $I: \mathds{R} \to \mathds{R}$ denote an increasing function. We say $I$ captures the right tail of random variable $X$ if
\begin{equation} \label{eq:right-tail-capture}
 \p{X > t} \leq \expp{- I(t)}, \quad \forall t > 0.
\end{equation}
\end{defn}


Note that for the moment $I(t)$ can be a generic function. However, as we will see later, in our theorems we will impose some constraints on $I(t)$. Clearly, $I_{br}(t) = - \log \p{X > t}$ captures the right tail of $X$ for any random variable $X$. We call $I_{br}(t)$ the basic rate capturing function. One can use $I_{br}(t)$ in our concentration results. However, as will be discussed later, it is often more convenient to approximate this basic tail capturing function. 

Given a sequence of independent and identically distributed random variables $X_1, X_2, \ldots, X_m$ with $\e{X_i}<\infty$, the goal of this paper is to study 
\[
\p{S_m - \e{S_m} > m t},
\]
where $S_m = \sum\limits_{i = 1}^m X_i$.  Based on the definition of the rate capturing function we state our concentration result. In the rest of the paper, we use the notation $X^L$ to denote the truncated version of the random variable $X$, i.e., 
\[
X^L = X \mathbb{I} (X \leq L). 
\]
\rii{
The following lemma plays a pivotal role in obtaining our main concentration result. 
\begin{lem} \label{lem:mgf-bound-general}
For all $\lambda > 0$ and $L>0$ we have
\begin{equation*} \label{eq:mgf-bound-general}
\log \e{\expp{\lambda (X^L - \e{X})}} \leq  \frac{k \intoo{L, \lambda}}{2} \lambda^2,
\end{equation*}
where
\begin{equation*}
k \intoo{L, \lambda} = \e{\intoo{X^L - \e{X}}^2 \mathds{I} \intoo{X^L \leq \e{X}}} + \e{\intoo{X^L - \e{X}}^2 \expp{\lambda \intoo{X^L - \e{X}}} \mathds{I} \intoo{X^L > \e{X}}}.
\end{equation*}
\end{lem}

\begin{proof}
From the mean value theorem we have
\begin{equation} 
\expp{\lambda X^L} = \expp{\e{\lambda X}} + \intoo{\lambda X^L - \e{\lambda X}} \expp{\lambda \e{X}} + \frac{1}{2} \intoo{\lambda X^L - \e{\lambda X}}^2 \expp{\lambda Y},
\end{equation}
where $Y$ is a random variable whose value is always between $\e{X}$ and $X^L$.  Hence,
\begin{align} \label{eq:proof main lemma 1}
\log \e{\exp\intoo{\lambda  X^L}}  =
\lambda \e{X} +  \log \intoo{1 + \lambda \intoo{\e{X^L} - \e{X}} + \frac{1}{2} \lambda^2  \e{\intoo{X^L - \e{X}}^2 \expp{\lambda Y - \e{\lambda X}} }}.
\end{align}

Note that $X^L  - X \leq 0$ and $\lambda > 0$. Thus,

\begin{align} \nonumber
\log \e{\exp\intoo{\lambda \intoo{  X^L - \e{X} } }} & =
\log \intoo{1 + \lambda \intoo{\e{X^L} - \e{X}} + \frac{1}{2} \lambda^2  \e{\intoo{X^L - \e{X}}^2 \expp{\lambda Y - \e{\lambda X}} }}
\\ & \leq \nonumber
\log \intoo{1 + \frac{1}{2} \lambda^2  \e{\intoo{X^L - \e{X}}^2 \expp{\lambda Y - \e{\lambda X}} }}
\\ & \leq \label{eq:log-mgf-bound-proof-in-term-of-y}
\frac{1}{2} \lambda^2  \e{\intoo{X^L - \e{X}}^2 \expp{\lambda Y - \e{\lambda X}} }.
\end{align}
Since $Y$ falls between $\e{X}$ and $X^L$ we have
\begin{equation*}
Y \leq \e{X} \mathds{I} \intoo{X^L \leq \e{X}} + X^L \mathds{I} \intoo{X^L > \e{X}}.
\end{equation*}
Hence the expectation in \eqref{eq:log-mgf-bound-proof-in-term-of-y} is bounded by
\begin{equation*}
\e{\intoo{X^L - \e{X}}^2 \mathds{I} \intoo{X^L \leq \e{X}}} + \e{\intoo{X^L - \e{X}}^2 \expp{\lambda \intoo{X^L - \e{X}}} \mathds{I} \intoo{X^L > \e{X}}}.
\end{equation*}
\end{proof}
}

Lemma \ref{lem:mgf-bound-general} enables us to prove the following concentration result that is the main contribution of this paper.

\begin{thm}[General Concentration] \label{thm:concentration-general}
Suppose $X_1, ..., X_m \overset{d}{=} X$ are independent and identically distributed random variables whose right tails are captured by an increasing and continuous function $I: \mathds{R} \rightarrow \mathds{R}^{\geq 0}$ with the property $I(t) = O(t)$ as $t \rightarrow \infty$.   For the sake of notational compactness define $Z^L \triangleq X^L- \e{X}$ and $v(L, \beta) = k \intoo{L,  \beta \frac{I(L)}{L}}$,  for $\beta \in (0, 1]$.    That is
\begin{align} \label{eq:bdd by c assumption} 
v \intoo{L, \beta} \triangleq \e{\intoo{Z^L}^2 \mathds{I} \intoo{Z^L \leq 0} + \intoo{Z^L}^2 \expp{\beta \frac{I(L)}{L} Z^L} \mathds{I} \intoo{Z^L > 0}}.
\end{align}
Finally, define $t_{\max}(\beta) \define \sup \cbr{t \geq 0: \; t \leq \beta v \intoo{mt,\beta} \frac{I(mt)}{mt} }$.\footnote{We set $t_{\max}=0$, when the set is empty. } Then,
\begin{equation} \label{eq:concentration ineq}
\p{S_m - \e{S_m} > m t} \leq \begin{dcases}
\expp{-c_t \beta I(mt)} + m \expp{-I(mt)}, & 
t \geq t_{\max}(\beta),
\\
\expp{-\frac{m t^2}{2 v \intoo{mt_{\max}(\beta),\beta}} } + m \expp{-  \frac{m t_{\max}(\beta)^2}{\beta v \intoo{m t_{\max}(\beta),\beta}} }, &
0 \leq t < t_{\max}(\beta),
\end{dcases} 
\end{equation}
where $c_t$ is a constant between $\frac{1}{2}$ and $1$.  More precisely, $c_t = 1 - \frac{1}{2} \frac{\beta v \intoo{mt,\beta}}{t} \frac{I(mt)}{mt}$. 
\end{thm}

The proof of this theorem can be found in Section \ref{sec:proofs}. Note that the concentration result we obtain is similar to the concentration results that exist for subexponential random variables; there is a region for $t$ in which the distribution of the sum looks like a Gaussian, and a second region in which the sum has heavier tail than a Gaussian. We will apply our theorem to some popular examples, including the subexponential distributions later. Before that, let us discuss some of the main features of this theorem.

\begin{rem} \label{rmk:choice of beta}
In Theorem \ref{thm:concentration-general}, ideally one would like to choose $\beta = 1$ to obtain the best rate in the first regime.  However,  it is usually the case that for bounding \eqref{eq:bdd by c assumption} one needs to choose $\beta < 1$.   In this case,  the larger the $\beta$,  the tighter bound one can achieve with Theorem \ref{thm:concentration-general}.
\end{rem}

\begin{rem} \label{rem:concentration-thm-with-upper-bound-for-c}
As is clear from the proof of Theorem \ref{thm:concentration-general}, one can replace $v \intoo{mt, \beta}$ with an upper bound.  In other words, if $v \intoo{mt, \beta} \leq \bar{v}$, then Theorem \ref{thm:concentration-general} remains valid by replacing $v \intoo{mt, \beta}$ with $\bar{v}$ in the definition of $t_{\max}$ and the coefficients appeared in \eqref{eq:concentration ineq}.
\end{rem}

\riii{
\begin{rem} \label{rmk:one regime ineq}

Equation \eqref{eq:concentration ineq} offers the best possible result implied by our analysis.  It is possible to obtain a single bound by adding the bounds for each region.  Below we obtain such a result. Consider the same settings as the ones in Theorem \ref{thm:concentration-general}. For notational simplicity, assume that the random variables are centered, i.e. $\e{X_i} = 0$. Then, for any $t \geq 0$ we have:

\begin{equation}
    \label{eq:concentration one regime}
    \p{S_m > m t} \leq 
    \expp{-\frac{m t^2}{2 v(mt, \beta)}} + \expp{-\beta \max \cbr{c_t, \frac{1}{2}} I(m t)} + m \expp{-I(mt)},
\end{equation}
where $c_t = 1 - \frac{\beta v(mt, \beta) I(mt)}{2 m t^2}$. We will prove \eqref{eq:concentration one regime} in Section \ref{sec:proof:remk one regime}.
\end{rem}
}

\begin{rem} \label{rmk:no free parameter}
Note that unlike many existing results in the literature, Theorem \ref{thm:concentration-general} does not leave any free parameters for the reader to tune.   The optimal values of all the parameters in the intermediate steps have been set. We will clarify this point and compare our result with some existing work in Section \ref{sec:compall}.    
One can use the following steps to obtain an exponential bound for the probability of interest:  
\begin{enumerate}
\item 
Obtain an upper bound for $v \intoo{L, \beta}$ defined in \eqref{eq:bdd by c assumption}.  (call it $\overline{v \intoo{L, \beta}}$)

\item
Plug in $L = mt$ in the upper bound of $v \intoo{L, \beta}$ and for $t > 0$ determine the region for which $t \leq \beta \overline{v \intoo{mt,\beta}} \frac{I(mt)}{mt} $.  Set $t_{\max}(\beta)$ as the supremum of this region.

\item 
Apply Theorem \ref{thm:concentration-general}.
($v \intoo{L, \beta}$ can be replaced by the upper bound $\overline{v \intoo{L, \beta}}$ for all values of $L$,  including $L = mt$ and $L = m t_{\max}(\beta)$)
\end{enumerate}

Note that step 2 can be skipped if one uses Remark \ref{rmk:one regime ineq} instead of Theorem \ref{thm:concentration-general}. Furthermore,  in Lemmas \ref{lem:exact-form-for-cl-second-term} and \ref{lem:cl-upper-bound}  we derive simple (easy-to-calculate) and yet sharp  upper bounds for $\nu(L, \beta)$. In the rest of this article, we study a few popular examples of heavy-tailed distributions, and show that following these simple steps lead to asymptotically sharp bounds.
\end{rem}

\riii{
\begin{rem} \label{rem:lower bound for finite sample}
Let $I_{br}(t) = - \log \p{X > t}$ be the basic rate capturing function.  By applying Remark \ref{rmk:one regime ineq} with this function we obtain
\begin{equation} \label{eq:upper buond for rem}
\p{S_m - \e{S_m} > m t} \leq 
\expp{-\frac{mt^2}{2 v(mt, \beta)}} + \expp{- \beta \max(c_t, 0.5) I_{br}(mt)} + m \expp{-I_{br}(mt)},
\end{equation}

for any positive $t$.  On the other hand,
\begin{align} 
\nonumber
    \p{S_m - \e{S_m} > m t} & \geq \sum_{i = 1}^m \p{X_i > mt} \p{X_j \leq mt \; \forall j \neq i, \quad S_m - X_i \geq \e{S_m}}
    \\ \nonumber & =
    m \expp{- I_{br}(mt)} \p{X_j \leq mt \ \; \forall j \leq m - 1, \quad S_{m - 1} - \e{S_{m -1}} \geq \e{X}}
    \\ & \label{eq:lower bound finite smaple}
     = \frac{1}{2} m \expp{- I_{br}(mt)} (1 + o(1)).
\end{align}

The last equality is obtained from the fact that $\p{X_i \leq mt} \to 1, \p{\frac{1}{\sqrt{m}}(S_{m - 1} - \e{S_{m -1}}) \geq \frac{\e{X}}{\sqrt{m}}} \to \frac{1}{2}$ according to the central limit theorem.
We will discuss in Section \ref{sec:large deviation} that $m \expp{-I_{br}(mt)}$ is the term that determines the asymptotic value of $\p{S_m - \e{S_m} > mt}$, i.e. $\lim\limits_{m \to \infty}\frac{\log \p{S_m - \e{S_m} > mt}}{\log \intoo{m \expp{-I_{br}(mt)}}} = 1$.
\end{rem}
}

Obtaining an accurate upper bound for $v \intoo{L,\beta}$ is a key to using Theorem \ref{coro:cbeta-for-poly-tail} for different applications. Since, we are often interested in the behavior of $v \intoo{mt, \beta}$ for large values of $mt$, it is usually instructive to understand the behavior of $v \intoo{L,\beta}$ for large values of $L$. Suppose that there exists a function $g(X)$ such that
\[
 \abs{\intoo{Z^L}^2 \mathds{I} \intoo{Z^L \leq 0} + \intoo{Z^L}^2 \expp{\lambda {Z^L}} \mathds{I} \intoo{Z^L > 0}} < g(X),
\]
and that $\e{g(X)} <\infty$. Further, assume that $I(L) = o(L)$ \rii{(hence $\lambda = \frac{\beta I(L)}{L} \to 0$, as $L \to \infty$)}. Then, from the dominated convergence theorem we have
\[
\limsup_{L \rightarrow \infty}  \e{\intoo{Z^L}^2 \mathds{I} \intoo{Z^L \leq 0} + \intoo{Z^L}^2 \expp{\lambda {Z^L}} \mathds{I} \intoo{Z^L > 0}} =  \e{(X- \e{X})^2}.
\]
Hence, if the random variables have bounded variances, then we expect $v \intoo{L,\beta} < \infty$ for all values of $L$. If we replace $v \intoo{mt, \beta}$ in Theorem \ref{thm:concentration-general} with a fixed number, then the statement of the theorem becomes simpler. Note that this argument is based on an asymptotic argument and is not particularly useful when we want to derive concentration bounds. Hence, our next few lemmas obtain simpler integral forms for $\e{ \intoo{Z^L}^2 \expp{\lambda {Z^L}} \mathds{I} \intoo{Z^L > 0}}$. 

\begin{lem} \label{lem:exact-form-for-cl-second-term}
Let $Z^L = X^L- \e{X}$, and $I_{br}(t) = - \log \p{X > t}$ denote the basic tail capturing function. Then,
\begin{equation*}
\e{ \intoo{Z^L}^2 \expp{\lambda {Z^L}} \mathds{I} \intoo{Z^L > 0}} = \int_0^{L - \e{X}} \expp{ \lambda t - I_{br}(t + \e{X}) } \intoo{2 t + \lambda t^2} dt.
\end{equation*}
\end{lem}

\begin{proof}
We have
\begin{align*}
\mathbb{E} & \left[  \intoo{Z^L}^2  \expp{\lambda {Z^L}} \mathds{I} \intoo{Z^L > 0} \right]
\\ & =
\int_0^\infty \p{ \intoo{Z^L}^2 \expp{\lambda {Z^L}} > u, Z^L > 0} du
\\ & =
\int_0^{L - \e{X} } \p{X > t + \e{X}} du, \qquad t^2 \expp{\lambda t} = u,
\\ & =
\int_0^{L - \e{X} } \expp{- I_{br} (t + \e{X})} \intoo{2t + \lambda t^2} \expp{ \lambda t} dt
\\ & =
\int_0^{L - \e{X} } \expp{ \lambda t - I_{br} (t + \e{X})} \intoo{2t + \lambda t^2}  dt,
\end{align*}
which completes the proof. 
\end{proof}

We can use the integral expression we derived in Lemma \ref{lem:exact-form-for-cl-second-term}, and the specific properties of the rate function that we have, to obtain simpler upper bounds for $\e{ \intoo{Z^L}^2 \expp{\lambda {Z^L}}\mathds{I} \intoo{Z^L > 0}}$. The following simple lemma is an upper bound we will use in our examples. 

\begin{lem} \label{lem:cl-upper-bound}
Suppose that $\frac{I(t)}{t}$ is a nonincreasing function, and let $\lambda = \frac{\beta I(L)}{L}$ Then,
\begin{eqnarray*}
\lefteqn{\e{ \intoo{Z^L}^2 \expp{\lambda {Z^L}} \mathds{I} \intoo{Z^L > 0}}} \nonumber \\
&\leq& 
\expp{ -\beta \e{X} \frac{I(L)}{L}}  \int_0^{L - \e{X}} \expp{- (1 - \beta) I(t + \e{X })} \intoo{2 t + \beta \frac{I(L)}{L} t^2} dt
\\ & \leq&
\expp{-\beta \e{X} \frac{I(L)}{L}} \int_0^{L - \e{X} } \expp{- (1 - \beta) I(t + \e{X})}\intoo{2t + \beta t I(t)} dt.
\end{eqnarray*} 
\end{lem}

\begin{proof}
Similar to the proof of Lemma \ref{lem:exact-form-for-cl-second-term}, we have
\begin{align*}
\mathbb{E} & \left[  \intoo{Z^L}^2  \expp{\lambda {Z^L}} \mathds{I} \intoo{Z^L > 0} \right]
\\ & \leq
\int_0^{L - \e{X} } \expp{- I (t + \e{X})} \intoo{2t + \lambda t^2} \expp{ \beta \frac{I(L)}{L} t} dt
\\ & \leq
\expp{-\beta \e{X} \frac{I(L)}{L}} \int_0^{L - \e{X} } \expp{- (1 - \beta) I(t + \e{X})}\intoo{2t + \beta \frac{I(L)}{L} t^2} dt
\\ & \leq
\expp{-\beta \e{X} \frac{I(L)}{L}} \int_0^{L - \e{X} } \expp{- (1 - \beta) I(t + \e{X})}\intoo{2t + \beta t I(t)} dt,
\end{align*}
where to obtain the last two inequalities we used the fact that $\frac{I(t)}{t}$ is a nonincreasing function. 
\end{proof}

We will later show, \rii{in Section \ref{sec:applications of theorem 1},} how combining Theorem \ref{thm:concentration-general} and Lemma \ref{lem:cl-upper-bound} leads to sharp concentration results for some well-known tail capturing functions.   It is straightforward to see that as long as $(1 - \beta)I(t + \e{X}) > 2 a \log t$ for some $a > 1$, the upper bound given by Lemma \ref{lem:cl-upper-bound} remains bounded even when $L \rightarrow \infty$. Hence, we can use these upper bounds for a broad range of tail decays. We will discuss this in more details at the end of this section.

\begin{rem}
 Note that Theorem \ref{thm:concentration-general} considers the case where $I(L) = O(L)$. The other cases, i.e.  $I(L) = \Omega (L)$, can be studied using standard arguments based on the moment generating function and hence, are not explored in this paper. We will later emphasize that this condition is not enough for the usefulness of Theorem \ref{thm:concentration-general}. For instance, if the tail is too heavy then $v \intoo{L,\beta}$ will be infinite. We will discuss this issue in more details later. 
\end{rem}

 Let us now show how Theorem \ref{thm:concentration-general} can be used in a few concrete examples which are popular in application areas. Our first example considers the well-studied class of subexponential distributions.

\subsection{Applications of Theorem \ref{thm:concentration-general}} \label{sec:applications of theorem 1}

\begin{coro} \label{coro:cbeta-for-poly-tail}
Let $I(t) = k t$ for some fixed coefficient $k$.  Then, for all $\beta \in (0,1)$ and $L > \e{X}$ we have
\begin{equation}
v \intoo{L, \beta} \leq \e{(X - \e{X})^2 \mathds{I} (X \leq \e{X}) } + \frac{1}{(1 - \beta)^3} \frac{2 }{k^2 \expp{ k \e{X}}} = v_\beta.
\end{equation}
 Hence for $m > \frac{\e{X}}{\beta v \intoo{\beta} k} $,
\begin{equation} \label{eq:coro:concentration ineq}
\p{S_m - \e{S_m} > m t} \leq \begin{dcases}
\expp{-c_t \beta k mt} + m \expp{-k m t}, & 
t \geq \beta v_\beta k,
\\
\expp{-\frac{1}{2 v_\beta} m t^2} + m \expp{-  \beta v_\beta k^2 m }, &
0 \leq t < \beta v_\beta k,
\end{dcases} 
\end{equation}
where $c_t = 1 - \frac{1}{2} \frac{\beta v_\beta k}{t}$.
\end{coro}
\begin{proof}
We would like to use Theorem \ref{thm:concentration-general} for proving the concentration. Toward this goal, we use Lemma \ref{lem:cl-upper-bound} to obtain an upper bound for $v \intoo{L, \beta}$. First note that, $\lambda = \beta \frac{I(L)}{L} = \beta k$. Hence, according to Lemma \ref{lem:cl-upper-bound} we have
\begin{eqnarray*}
\lefteqn{\e{ \intoo{Z^L}^2 \expp{\lambda {Z^L}} \mathds{I} \intoo{Z^L > 0}}} \nonumber \\
&\leq& 
\expp{{-}\beta k \e{X} }  \int_0^{L - \e{X}} \expp{- (1 - \beta) k(t + \e{X })} \intoo{2 t + \beta k t^2} dt \nonumber \\
&\leq& \expp{{-}\beta k \e{X} }  \int_0^{\infty} \expp{- (1 - \beta) k(t + \e{X })} \intoo{2 t + \beta k t^2} dt \nonumber \\
&=& \expp{-k \e{X}} \int_0^\infty \expp{- (1 - \beta) k t} \intoo{\beta k t^2 + 2 t} dt
\\ & = &
\expp{-k \e{X}} \intoo{ \frac{\beta k}{(1 - \beta)^3 k^3} \Gamma(3) + \frac{2}{(1 - \beta)^2 k^2} \Gamma(2) }
\\ & = &
\expp{-k \e{X}} \frac{2}{k^2} \frac{1}{(1 - \beta)^3}.
\end{eqnarray*}
We also have that if $L>\e{X}$, then 
\begin{eqnarray*}
\lefteqn{\e{\intoo{Z^L}^2 \mathds{I} \intoo{Z^L \leq 0}}  = \e{\intoo{X^L-\e{X}}^2 \mathds{I} \intoo{X^L \leq \e{X}}}} \nonumber \\
&=& \e{\intoo{X-\e{X}}^2 \mathds{I} \intoo{X \leq \e{X}}}. 
\end{eqnarray*}
\end{proof}

Our next example considers subWeibull distributions. 

\begin{coro} \label{coro:concentration-with-cl-bound-finite-sample}
Let $X$ be a centered random variable, i.e. $\e{X} = 0$, whose tail is captured by $c_\alpha \sqrt[\alpha]{t}$ for some $\alpha \geq 1$.  Moreover, assume $\e{X^2 \mathds{I}(X \leq 0)} = \sigma_-^2 < \infty$.  Then, we have
\begin{equation*}
v \intoo{L, \beta } \leq \bar{v}\intoo{L, \beta } \triangleq \sigma_-^2 + \frac{\Gamma(2 \alpha + 1)}{\intoo{(1 - \beta) c_\alpha}^{2 \alpha}}  + L^{\frac{1}{\alpha} - 1} \frac{\beta c_\alpha \Gamma(3 \alpha + 1)}{3 \intoo{(1 - \beta) c_\alpha}^{3 \alpha}} .
\end{equation*}
Hence, Theorem  \ref{thm:concentration-general} can be applied with $\bar{v} \intoo{mt, \beta}$ and  $t_{\max}(\beta) = \intoo{\beta \bar{v} \intoo{mt, \beta} c_\alpha}^{\frac{\alpha}{2 \alpha - 1}} m ^{- \frac{\alpha - 1}{2 \alpha - 1}}$.
\riii{Furthermore, using the bound in Remark \ref{rmk:one regime ineq} we obtain
\begin{equation*}
    \p{S_m > m t} \leq \expp{-\frac{m t^2}{2 \bar{v}\intoo{L, \beta }}} + \expp{- \beta \max(c_t, 0.5) c_\alpha \sqrt[\alpha]{mt}} + m \expp{-c_\alpha \sqrt[\alpha]{mt}}, \quad \forall t \geq 0, 
\end{equation*}
where $\beta < 1$ is arbitrary and $\bar{v}\intoo{L, \beta }$ is the upper bound of $v(mt, \beta)$, and $c_t = 1 - \frac{\beta \bar{v}\intoo{L, \beta } c_\alpha \sqrt[\alpha]{mt}}{2 m t^2}$.}
\end{coro}

\begin{proof}
Note that since $\alpha \geq 1$, $\frac{I(t)}{t}$ is indeed nonincreasing. We just need to apply Lemma \ref{lem:cl-upper-bound} with $I(t) = c_\alpha  \sqrt[\alpha]{t}$ to obtain
\begin{align*}
\int_0^L & \expp{- (1 - \beta) c_\alpha \sqrt[\alpha]{t}} \intoo{2 t + \beta c_\alpha L^{\frac{1}{\alpha} - 1} t^2} dt
\\ & \leq
\int_0^\infty \expp{-u} \intoo{\frac{2 u^\alpha}{\intoo{(1 - \beta) c_\alpha}^\alpha} + \frac{\beta c_\alpha L^{\frac{1}{\alpha} - 1} u^{2 \alpha}}{\intoo{(1 - \beta) c_\alpha}^{2 \alpha}}} \frac{\alpha u^{\alpha - 1}}{\intoo{(1 - \beta) c_\alpha}^\alpha} du
\\ & =
\frac{2 \alpha}{\intoo{(1 - \beta) c_\alpha}^{2 \alpha}} \Gamma(2 \alpha) + \frac{\beta c_\alpha L^{\frac{1}{\alpha} - 1} \alpha}{\intoo{(1 - \beta) c_\alpha}^{3 \alpha}} \Gamma(3 \alpha)
\\ & =
\frac{\Gamma(2 \alpha + 1)}{\intoo{(1 - \beta) c_\alpha}^{2 \alpha}}  + L^{\frac{1}{\alpha} - 1} \frac{\beta c_\alpha \Gamma(3 \alpha + 1)}{3 \intoo{(1 - \beta) c_\alpha}^{3 \alpha}}.
\end{align*}
Finally, it is straightforward to note that
\[
\e{(X^L)^2  \mathds{I}(X^L \leq 0)} \leq \e{X^2 \mathds{I}(X \leq 0)}. 
\]
\end{proof}

In our last example, we consider random variables with polynomially decaying tails. 

\begin{coro} \label{coro:cl-upper-bound-poly-tail}
Let $X$ be a centered random variable, i.e. $\e{X} = 0$, whose tail is captured by $\gamma \log t$, where $\gamma > 2$. Moreover, assume $\e{X^2 \mathds{I}(X \leq 0)} = \sigma_-^2 < \infty$. Then, we have
\begin{equation} \label{eq:cl-upper-bound-poly-tail}
v \intoo{L, \beta} \leq \bar{v} \intoo{L, \beta} \define
\begin{dcases}
 \sigma_-^2 + L^\frac{\gamma \beta}{L} +   \frac{2 - \frac{\gamma \beta }{2 - \gamma(1 - \beta)}}{2 - \gamma(1 - \beta)} \intoo{L^{2 - \gamma(1 - \beta)} - 1}+ \frac{\gamma \beta L^{2 - (1 - \beta) \gamma} \log L}{2 - (1 - \beta) \gamma},
 &
\beta \neq 1 - \frac{2}{\gamma}, 
\\
\sigma_-^2 + L^\frac{\gamma - 2}{L} + 2 \log L + \frac{\gamma - 2}{2} \intoo{\log L}^2, &
\beta = 1 - \frac{2}{\gamma}.  
\end{dcases}
\end{equation}

\riii{
Setting $\bar{v}= \bar{v}(mt, \beta)$ in Remark \ref{rmk:one regime ineq}, we have that for any $\beta < 1$:

\begin{equation}
    \label{eq:coro poly tail bound with bar v}
    \p{S_m > mt} \leq \expp{- \frac{mt^2}{2 \Bar{v}}} + \frac{1}{(mt)^{
    \beta \max(c_t, 0.5) \gamma}} + \frac{m}{(mt)^{\gamma}},
    \qquad \forall t \geq 0
\end{equation}
where $c_t = 1 - \frac{\beta \Bar{v} \gamma \log(mt)}{2 mt^2}$.
Note that for a fixed $\beta < 1 - \frac{2}{\gamma}, \; \Bar{v}$ remains bounded as $mt \to \infty$, and for $\beta < 1 - \frac{1}{\gamma}$, while $\Bar{v}$ grows, the first term of \eqref{eq:coro poly tail bound with bar v} remains negligible compared to the other polynomial terms.   Note that Theorem \ref{thm:concentration-general} also offers an upper bound of form 
\begin{equation}
    \label{eq:our bound in poly tail simplified}
    \p{S_m > mt} \leq \frac{1}{(mt)^{c_t \beta \gamma}} + \frac{m}{(mt)^\gamma}, \qquad t \geq t_{\max}(\beta).
\end{equation}
}

\vspace{.25 cm}
\begin{proof}
Note that
\begin{eqnarray*}
\lefteqn{\e{\intoo{X^L}^2 \expp{\lambda X^L} \mathds{I}(X^L > 0)}  =
\e{\intoo{X^L}^2 \expp{\lambda X^L} \mathds{I}(0 < X^L \leq 1)} }
\\ & \;\;\; +&
\e{\intoo{X^L}^2 \expp{\lambda X^L} \mathds{I}(X^L > 1)}
\\ & \leq&
\expp{\beta \gamma \frac{\log(L)}{L}} + \e{\intoo{X^L}^2 \expp{\lambda X^L} \mathds{I}(X^L > 1)}
\\ & =&
L^{\frac{\gamma \beta}{L}} + \e{\intoo{X^L}^2 \expp{\lambda X^L} \mathds{I}(X^L > 1)}.
\end{eqnarray*}
Thus, for $\beta \neq 1 - \frac{2}{\gamma}$, using the upper bound given in Lemma \ref{lem:cl-upper-bound}, we just need to show
\begin{equation} 
\int_1^{L} \expp{- (1 - \beta) \gamma \log t}\intoo{2t + \beta \gamma t \log t} dt =
\frac{2 - \frac{\gamma \beta }{2 - \gamma(1 - \beta)}}{2 - \gamma(1 - \beta)} \intoo{L^{2 - \gamma(1 - \beta)} - 1}+ \frac{\gamma \beta L^{2 - (1 - \beta) \gamma} \log L}{2 - (1 - \beta) \gamma}.
\end{equation}
Toward this goal, note that
\begin{eqnarray*}
\lefteqn{\int_1^{L} \expp{- (1 - \beta) \gamma \log t}\intoo{2t + \beta \gamma t \log t} dt  =
\int_1^{L} t^{1 - (1 - \beta) \gamma}\intoo{2 + \beta \gamma  \log t} dt}
\\  &=&
\frac{t^{2 - (1 - \beta) \gamma}}{2 - (1 - \beta) \gamma } \intoo{2 + \beta \gamma \intoo{-\frac{1}{2 - (1 - \beta) \gamma} + \log t}} \sVert[4]_1^L
\\ & = &
\frac{2 - \frac{\gamma \beta }{2 - \gamma(1 - \beta)}}{2 - \gamma(1 - \beta)} \intoo{L^{2 - \gamma(1 - \beta)} - 1}+ \frac{\gamma \beta L^{2 - (1 - \beta) \gamma} \log L}{2 - (1 - \beta) \gamma}.
\end{eqnarray*}
In the above equality, we are using $\int t^k = \frac{1}{k + 1} t^{k + 1}$ and $\int t^k \log t = \intoo{- \frac{1}{(k + 1)^2} + \frac{\log t}{k + 1}} t^{k + 1}$. \\

\noindent For $\beta = 1 - \frac{2}{\gamma} $ we have $1 - (1 - \beta) \gamma = -1$ and $\gamma \beta = \gamma - 2$. Hence
\begin{eqnarray*}
\lefteqn{\int_1^{L} \expp{- (1 - \beta) \gamma \log t}\intoo{2t + \beta \gamma t \log t} dt  =
\int_1^{L} t^{-1}\intoo{2 + (\gamma - 2)  \log t} dt}
\\  &= &
2 \log t + \frac{\gamma - 2}{2} \intoo{\log t}^2 \sVert[4]_1^L = 2 \log L + \frac{\gamma - 2}{2} \intoo{\log L}^2,
\end{eqnarray*}
which concludes the proof.
\end{proof}
\end{coro}

\begin{rem} 
Note that $\beta < 1 - \frac{2}{\gamma}$ is equivalent to $2 - (1 - \beta) \gamma < 0$. Hence, the right hand side of  \eqref{eq:cl-upper-bound-poly-tail} remains bounded when $L$ grows to infinity. By letting $\beta$ get closer to $0$ we can cover any $\gamma>2$. Hence, we can obtain a concentration inequality for the sum of independent and identically distributed random variables with polynomially decaying tails as long as $P(X>t) < \frac{1}{t^\gamma}$ for some $\gamma >2$. 
\end{rem}

Let us try to find another bound for $v \intoo{L,\beta}$ for the distributions we discussed in Corollaries \ref{coro:concentration-with-cl-bound-finite-sample} and \ref{coro:cl-upper-bound-poly-tail}. These bounds enable us to obtain another concentration result that is in some sense  sharper than the one we derived above and shows the flexibility of our framework. 

\begin{lem} \label{lem:limit-of-c}
Suppose that ${\rm var}(X) < \infty$ and the right tails of random variables $X$ is captured by $I(t)$.  Suppose that $I(t)$ and $\beta$ satisfy one of the following conditions:
\begin{enumerate}[(a)]
\item $I(t) = I_\alpha(t) = c_\alpha \sqrt[\alpha]{t}$ for $\alpha >1$ and $\beta < 1,$

\item $I(t) = \gamma \log t$ for $\gamma > 2$ and $\beta < 1 - \frac{2}{\gamma}$.

\end{enumerate}
Then, \milad{for $0 \leq \lambda_{L, \beta} \leq \beta \frac{I(L)}{L}$}  we have
\begin{equation*}
\lim_{L \to \infty} \e{\intoo{X^L - \e{X}}^2 \intoo{ \mathds{I} \intoo{X^L \leq \e{X}} + \expp{\lambda_{L, \beta} \intoo{X^L - \e{X} }} \mathds{I} \intoo{X^L > \e{X}} } }  = {\rm Var}(X).
\end{equation*}
\end{lem}
The proof of this lemma is presented in Section \ref{ssec:proof:lem1}. This lemma implies that if $L$ is large enough, then we should expect $v \intoo{L,\beta}$ to be very close to ${\rm Var}(X)$. So, assuming $mt$ is large enough we can obtain a more accurate concentration result. 

\begin{coro}\label{coro:weibul}
Suppose that the right tails of independent and identically distributed random variables $X_1, X_2, \ldots, X_m$ are captured by $c_\alpha \sqrt[\alpha]{t}$ for $\alpha >1$, and $Var(X_i) = \sigma^2$. Define $S_m = \sum\limits_{i \leq m} X_i$. Then, for any $0 < \beta < 1$ and $\epsilon > 0$, there is a constant $C_\epsilon$ such that for all $mt > C_\epsilon$

\begin{equation}
\p{S_m - \e{S_m} > mt} \leq 
\begin{dcases}
\expp{-c_t \beta c_\alpha \sqrt[\alpha]{mt}} + m \expp{-c_\alpha \sqrt[\alpha]{mt}},
&
t > t_{\max},
\\
\expp{- \frac{m t^2}{2 \intoo{\sigma^2 + \epsilon}} } + m \expp{- \frac{m t_{\max} ^2}{\beta( \sigma^2 + \epsilon )}},
&
t \leq t_{\max}, 
\end{dcases}
\end{equation}
where $t_{\max} = \intoo{\beta (\sigma^2 + \epsilon) c_\alpha}^{\frac{\alpha}{2 \alpha - 1}} {m}^{- \frac{\alpha - 1}{2 \alpha - 1}}$ and $c_t = 1 - \frac{1}{2} \beta (\sigma^2 + \epsilon) c_\alpha m^{\frac{1}{\alpha} - 1} t^{\frac{1}{\alpha} - 2} $ varies between $\frac{1}{2}$ and $1$.
\end{coro}

\begin{proof}
Note that, by Lemma \ref{lem:limit-of-c}, for any given $\epsilon > 0$ we can find a positive constant $C_\epsilon$, such that
\begin{equation*}
\e{\intoo{X^L - \e{X}}^2 \intoo{ \mathds{I} \intoo{X^L \leq \e{X}} + \expp{\lambda_{L, \beta} \intoo{X^L - \e{X} }} \mathds{I} \intoo{X^L > \e{X}} } }  \leq \sigma^2 + \epsilon, \quad \forall L > C_\epsilon.
\end{equation*}
Hence, for all $mt > C_\epsilon$, Theorem \ref{thm:concentration-general} is applicable with $v \intoo{L,\beta} = \sigma^2 + \epsilon$.  The corollary follows by substituting this $v \intoo{L,\beta}$ and $I(t) = c_\alpha \sqrt[\alpha]{t}$ in Theorem \ref{thm:concentration-general}.
\end{proof}

\begin{rem}
According to Corollary \ref{coro:weibul}, if $C_\epsilon < mt \leq m t_{\max}$, then $\p{S_m - \e{S_m} > mt}$ is upper bounded by $\expp{- \frac{m t^2}{2 \intoo{\sigma^2 + \epsilon}} } + m \expp{- \frac{m t_{\max} ^2}{\beta( \sigma^2 + \epsilon )}}$. Note that  $\expp{- \frac{m t^2}{2 \intoo{\sigma^2 + \epsilon}} } $ is very close to the term that appears in the central limit theorem. Furthermore, if $mt> t_{\max}$, then $\p{S_m - \e{S_m} > mt}$ is bounded from above by $\expp{-c_t \beta c_\alpha \sqrt[\alpha]{mt}} + m \expp{-c_\alpha \sqrt[\alpha]{mt}}$. Again we will show in the next section that this bound is sharp. Hence, an accurate bound for $v \intoo{L, \beta}$ results in an accurate concentration result.
\end{rem}

\begin{rem}
Using part (b) of Lemma \ref{lem:limit-of-c}, a corollary similar to Corollary \ref{coro:weibul} can be also written for $I(t) = \gamma \log t$ with $\gamma> 2$. For the sake of brevity, we do not repeat this corollary. Hence, Theorem \ref{thm:concentration-general} can be used to obtain concentration results as long as $I(t) > \gamma \log t$ with $\gamma> 2$ (for large enough values of $t$).  Note that if $I_{br}(t) = \gamma \log t$ for $\gamma<2$, then the variance of the random variable is unbounded. This is the region in which the sum of independent and identically distributed random variables does not converge to a Gaussian and it converges to other stable distributions (See Chapter 1 of \cite{nolan2003stable}). We leave the study of the concentration of sums of such random variables to future research.
\end{rem}

\rii{In order to illustrate the behavior of our upper bounds, we plot them for $X_i \sim \mathcal{N}(0, 1)^\kappa$ for $\kappa= 4, 5, 10$.  Note that in this case we have $\alpha = \kappa / 2$ and $I(t) = \frac{1}{2} \sqrt[\alpha]{t} - \log(2)$ is a rate capturing function. When $m t$ is large enough $I(m t)$ can be approximated with $\Tilde{I}(m t) = \frac{1}{2} \sqrt[\alpha]{m t}$.  Therefore with the terminology of Corollary \ref{coro:weibul} we have $c_\alpha = \frac{1}{2}$.  Below, we plot the upper bound functions in the asymptotic case of $\beta = c_t = 1, \epsilon = 0$ (the red curve is the plot of $\expp{- \frac{m t^2}{2 {\sigma^2 }} } + m \expp{- \frac{m t_{\max} ^2}{ \sigma^2}}$ and the blue curve is the plot of $(1 + m) \expp{-\frac{1}{2} \sqrt[\alpha]{mt}}$ which are the asymptotic limits of bounds offered by Corollary \ref{coro:weibul}).  Hence, $t_{\max} = \intoo{\frac{\sigma^2}{2}}^{\frac{\alpha - 1}{2 \alpha - 1}}$, where $\sigma^2 = {\rm Var}(\mathcal{N}(0, 1)^\kappa)$.  In addition to the upper bounds given by Corollary \ref{coro:weibul}, we sampled $10^7$ copies of $S_m = \sum\limits_{i=1}^m X_i$ and plotted the histogram of this data to achieve a numerical approximation for the tail and compare it with what our theory offers. 
}

\begin{figure}[h!]  
    \centering
    \begin{subfigure}[b]{0.3\textwidth}
    \includegraphics[width = \textwidth]{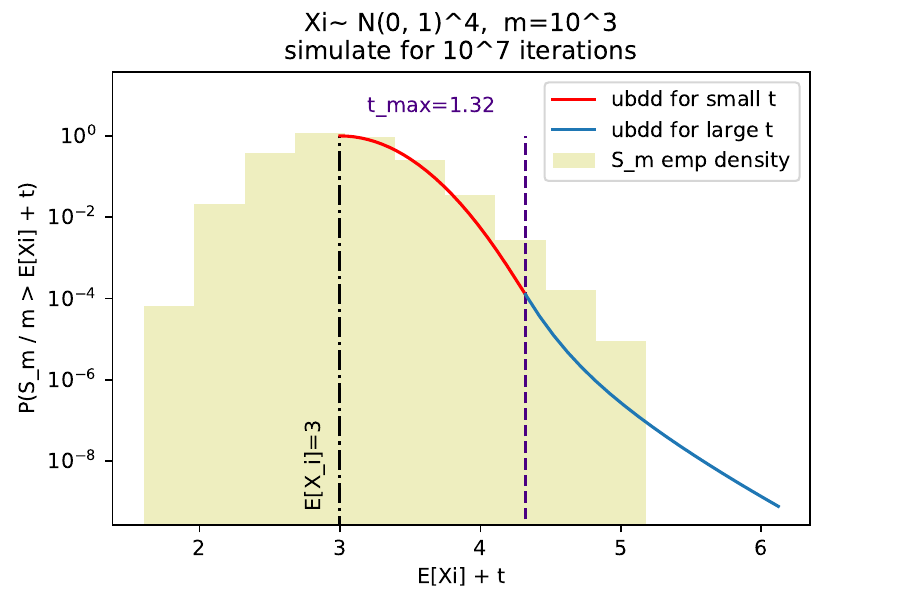}
    \end{subfigure}
    \begin{subfigure}[b]{0.3\textwidth}
    \includegraphics[width = \textwidth]{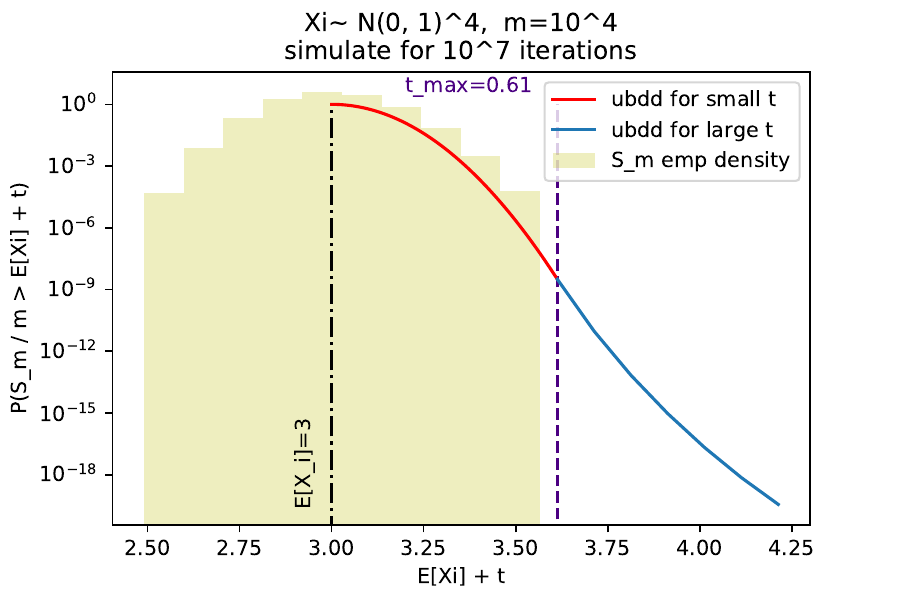}
    \end{subfigure}
    \begin{subfigure}[b]{0.3\textwidth}
    \includegraphics[width = \textwidth]{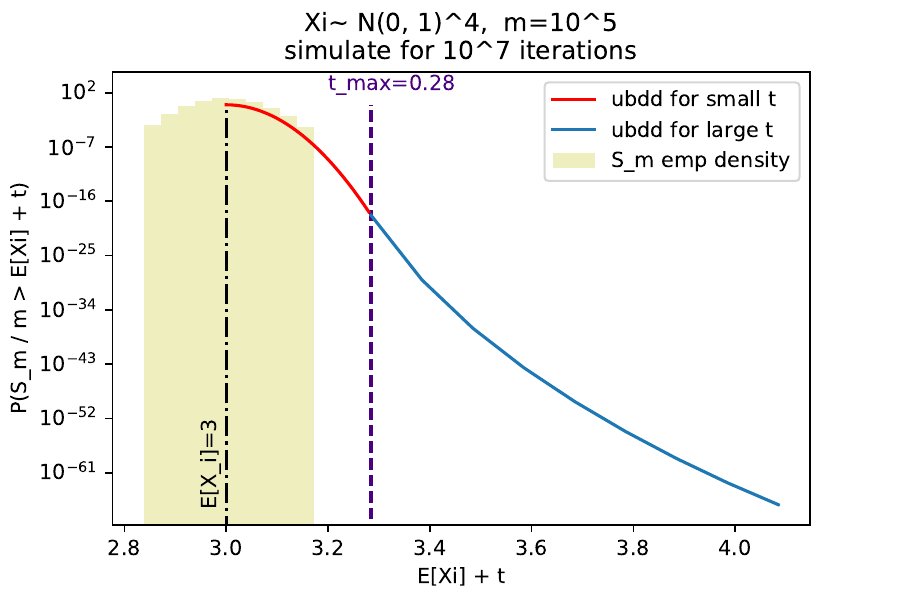}
    \end{subfigure}
    
    \begin{subfigure}[b]{0.3\textwidth}
    \includegraphics[width = \textwidth]{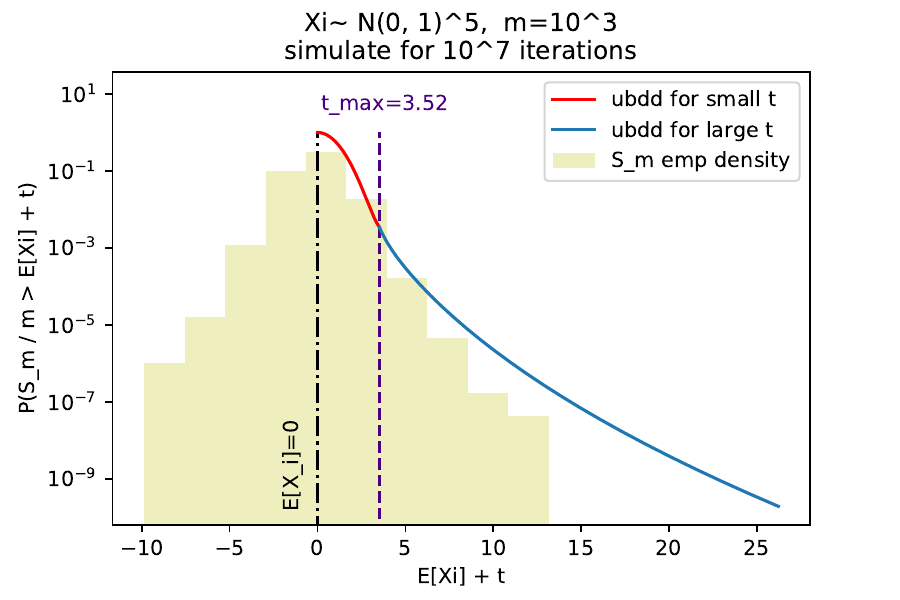}
    \end{subfigure}
    \begin{subfigure}[b]{0.3\textwidth}
    \includegraphics[width = \textwidth]{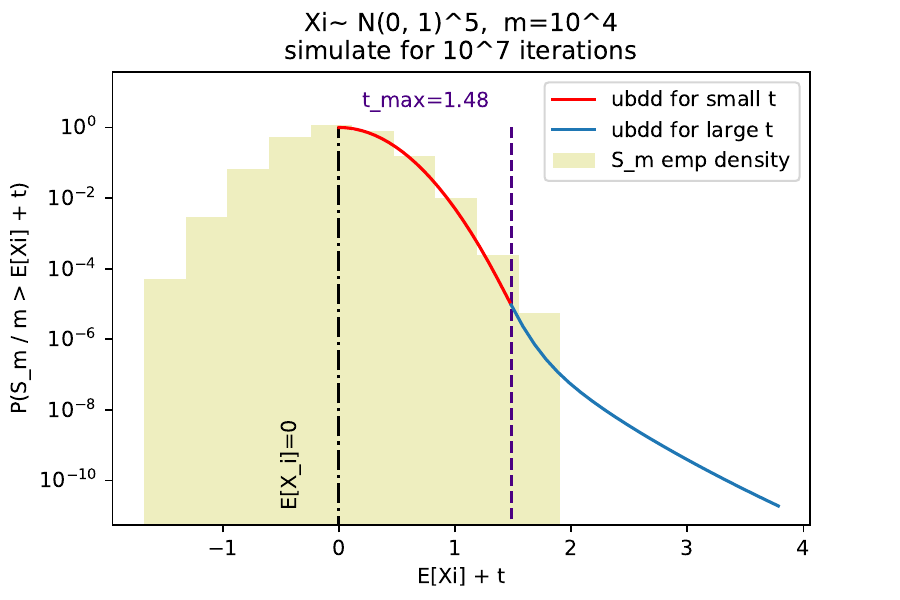}
    \end{subfigure}
    \begin{subfigure}[b]{0.3\textwidth}
    \includegraphics[width = \textwidth]{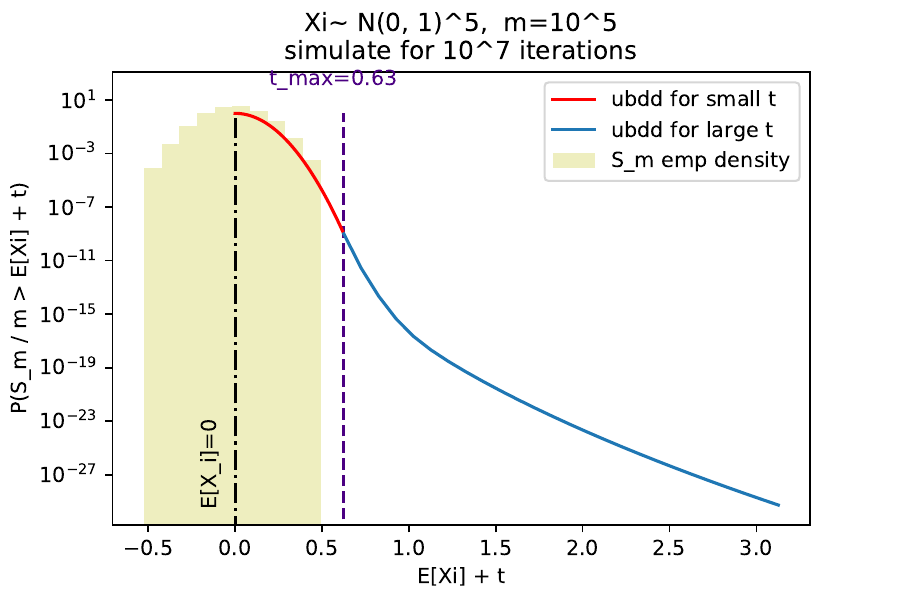}
    \end{subfigure}
    
    \begin{subfigure}[b]{0.3\textwidth}
    \includegraphics[width = \textwidth]{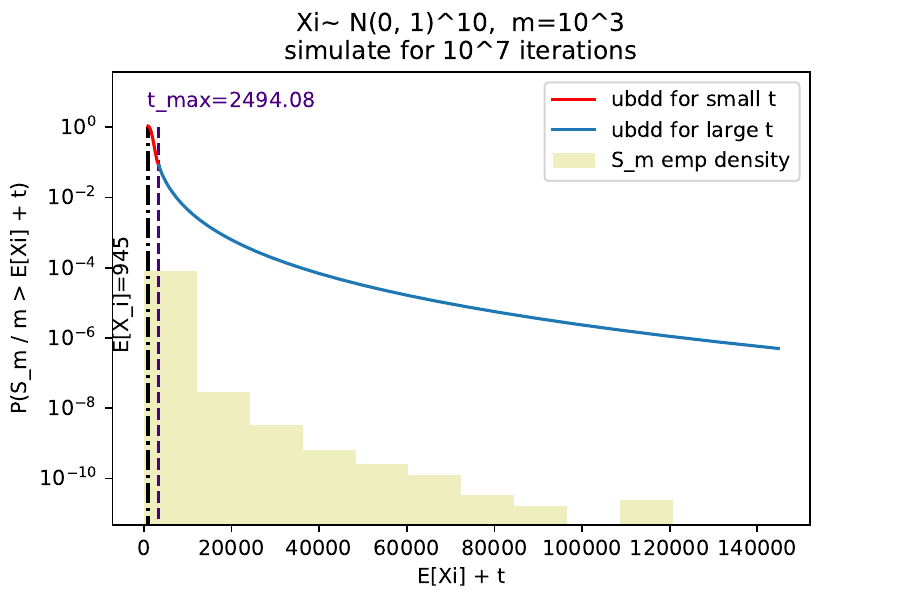}
    \end{subfigure}
    \begin{subfigure}[b]{0.3\textwidth}
    \includegraphics[width = \textwidth]{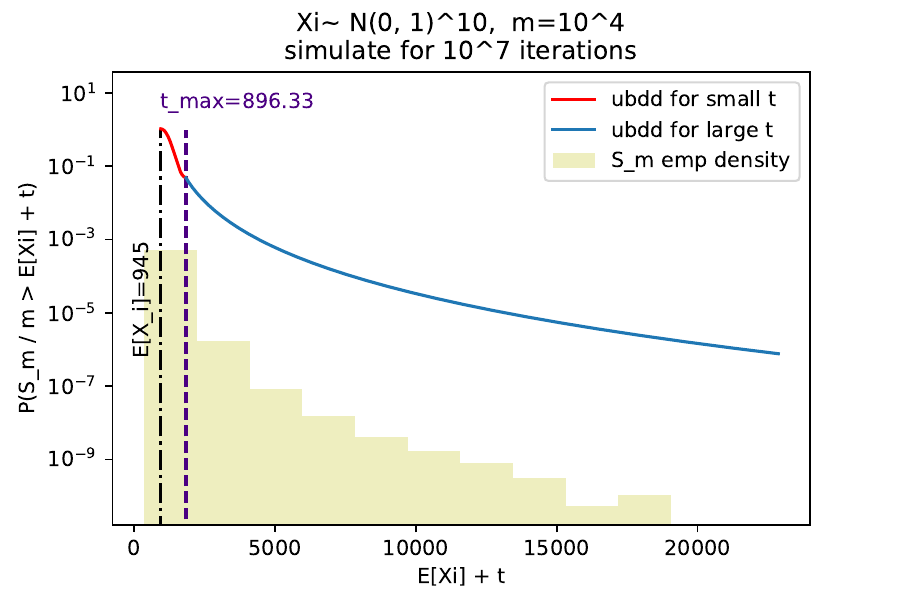}
    \end{subfigure}
    \begin{subfigure}[b]{0.3\textwidth}
    \includegraphics[width = \textwidth]{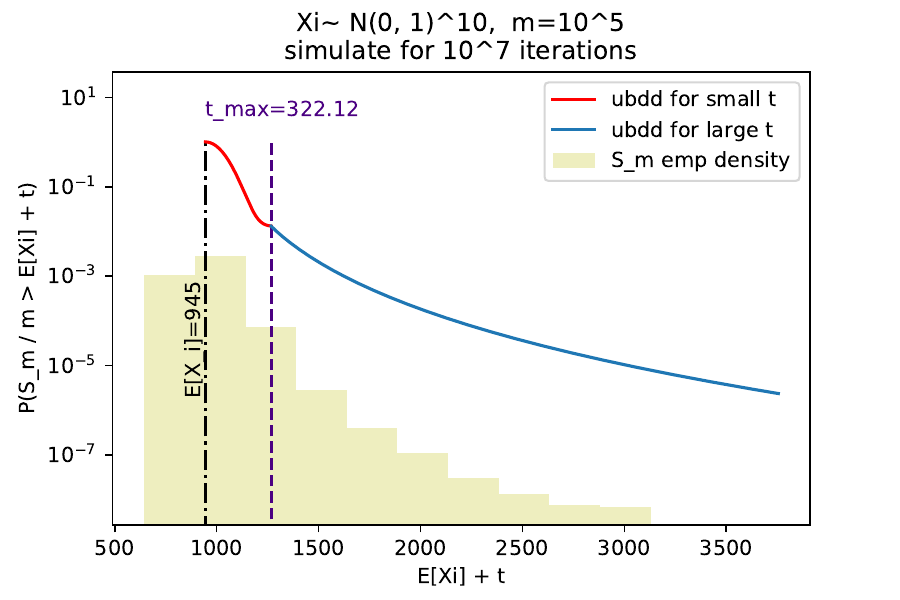}
    \end{subfigure}
    
    \label{fig:simulations}
    \caption{\centering Upper bounds in Corollary \ref{coro:weibul} for $X \sim \mathcal{N}(0, 1)^\kappa$,  $\kappa = 4, 5, 10$ and $m = 10^3, 10^4, 10^5$, 
    red curve is $\expp{- \frac{m t^2}{2 {\sigma^2 }} } + m \expp{- \frac{m t_{\max} ^2}{ \sigma^2}}$, blue curve is $(1 + m) \expp{-\frac{1}{2} \sqrt[\alpha]{mt}}$, 
    $t_{\max} = \intoo{\frac{\sigma^2}{2}}^{\frac{\alpha - 1}{2 \alpha - 1}}$}
\end{figure}

\subsection{Large deviation} \label{sec:large deviation}

In this section, as a simple byproduct of what we have proved for obtaining concentration bounds and also evaluating the sharpness of our results, we study the large deviation properties of the sums of independent and identically distributed random variables. Towards this goal, we consider the limiting version of Definition \ref{def:tail capture} in which the exact rate of decay of the tail is captured by $I(t)$. 

\begin{defn} \label{def:tail capture_l}
Let $I: \mathds{R} \to \mathds{R}$ denote an increasing function. We say $I$ captures the right tail of random variable $X$ in the limit if
\begin{equation} \label{eq:right-tail-capture_l}
\lim_{t \to \infty} \frac{- \log \intoo{\p{X > t}} }{I(t)} = 1.
\end{equation}
We say a random variable is superexponential if its tail is captured in limit by a function $I$ such that $I(t) = o(t)$ as $t\rightarrow \infty$.
\end{defn}

Note that if the basic right tail capturing function satisfies $I_{br}(t) = o(t)$, then the moment generating function of the distribution is infinity for $\lambda \in \intoo{0, \infty}$. Hence, Cramer's theorem is not useful. Our next theorem offers a sharp large deviation result for superexponential random variables.

\begin{thm}[General Large Deviation] \label{thm:ld-general}
Suppose that $X_1, X_2, \ldots, X_m$ are superexponential random variables with finite variance whose tails are captured in the limit by $I(t)$.  Furthermore, suppose that $I$ is an increasing function and $ \lim\limits_{t \to \infty} \frac{\log(t)}{I(t)} = 0$. Finally, let $\gamma_m$ be an increasing sequence of real numbers that satisfy
\begin{equation}
\log m \ll I(\gamma_m) \ll \frac{\gamma_m^2}{m}.\footnote{$f(t) \ll g(t)$ means that $f(t) = o(g(t))$ as $t \rightarrow \infty$. }
\end{equation}
  If \eqref{eq:bdd by c assumption} remains bounded for $X_1$ and  for all $\beta < 1$, then 
\begin{equation}
\lim_{m \to \infty} \frac{ - \log \p{S_m - \e{S_m} > \gamma_m} }{I(\gamma_m)} = 1.
\end{equation}
\end{thm}

The proof of this theorem is presented in Section \ref{sec:proof:theorem2}.  Again we use this theorem to obtain large deviation results for a few concrete examples. 

\begin{coro}\label{corr:subweibule:ldp}
Let the tail of independent and identically distributed random variables $X_1, X_2, \ldots, X_m$ be captured by $I(t) = a_\alpha \sqrt[\alpha]{t}$ in the limit, where $\alpha > 1$. Then, we have
\begin{equation*}
\lim_{m \to \infty} \frac{- \log \p{S_m - \e{S_m} > mt}}{\sqrt[\alpha]{m}} = a_\alpha \sqrt[\alpha]{t}.
\end{equation*}
\end{coro}

\begin{proof}
It suffices to choose $\gamma_m = mt$ and apply Theorem \ref{thm:ld-general} with $I(t) = a_\alpha \sqrt[\alpha]{t}$.  Note that
\begin{equation}\label{eq:subWeibul}
\log m \ll c_\alpha \sqrt[\alpha]{mt} \ll \frac{(mt)^2}{m} = m t^2, 
\end{equation}
for all $\alpha > 1$.
\end{proof}

\begin{rem} \label{rem:fill the gap of large deviation}
We should emphasize that the large deviation result for subWeibull distribution has been studied in the literature \cite{mikosch_ld}, \cite{borovkov2006integro}.
Being able to answer this question for subWeibull distributions, although it is not novel, shows the strength of the results developed in this paper.
Note that even if $t$ grows with $m$, as long as \eqref{eq:subWeibul} is satisfied, i.e. $mt_m \gg m^{\frac{\alpha}{2 \alpha -1}}$,  we have 
\begin{equation*}
\lim_{m \to \infty} \frac{- \log \p{S_m - \e{S_m} > mt_m}}{\sqrt[\alpha]{mt_m}} = a_\alpha.
\end{equation*}
On the other hand, it is known that if $m t_m \ll m^{\frac{\alpha}{2 \alpha - 1}} $, then the decay is characterized by $ \bar{\Phi} \intoo{\frac{m t_m}{\sqrt{m {\rm Var}(X)}} } $, where $\bar{\Phi} = 1 - \Phi$, and $\Phi$ denotes the cumulative distribution function of a standard normal random variable \cite{mikosch_ld}.  According to Table 3.1 of \cite{mikosch_ld} a similar result as the one presented in Corollary \ref{corr:subweibule:ldp} has been known for $ \gamma_m \gg m^{\frac{\alpha}{2 \alpha - 2}} $ when $ 0 \leq \frac{1}{\alpha} \leq \frac{1}{2} $. However as we discussed above, an extension of  Corollary \ref{corr:subweibule:ldp} fills the gap between $m^{\frac{\alpha}{2 \alpha - 2}}$ and $m^{\frac{\alpha}{2 \alpha - 1}}$, and shows that in this region still the tail of the sum behaves like the tail of the maximum.
\end{rem}

Theorem \ref{thm:ld-general} does not cover the polynomially-decaying tails. Hence, for the sake of completeness we discuss the polynomial example below.

\begin{coro} \label{coro:ld-poly-tail}
Suppose $X$ has zero mean and  finite variance $\sigma^2$ and its right tail is captured by $I(t) = \alpha \log t$ for some $\alpha > 2$.  For any sequence $\gamma_m$ that satisfies any of the following conditions
\begin{enumerate}[(i)]
\item 
 $\lim\limits_{m \to \infty} \frac{\log m}{\log \gamma_m} = k$ for some $k < 2$,
 
 \item  $\lim\limits_{m \to \infty} \frac{\log m}{\log \gamma_m} = 2$ and $\gamma_m \gg \sqrt{m \log m}$,
\end{enumerate}
 we have
\begin{equation}
\lim_{m \to \infty} \frac{- \log \p{S_m  > \gamma_m}}{I(\gamma_m) - \log m} = 1.
\end{equation}
\end{coro}

The proof can be found in Section \ref{ssec:proof:lastcor}. 

\begin{rem}
The result of Corollary \ref{coro:ld-poly-tail} is known in the literature. For instance, the interested reader may refer to Proposition 3.1 in  \cite{mikosch_ld}).  The main reason it is mentioned here is to show that this is also a simple byproduct of our main results in Section \ref{ssec:concentration}. Note that the conditions Corollary \ref{coro:ld-poly-tail} imposes on the growth of $\gamma_m$ cover all sequences that satisfy $\gamma_m \gg \sqrt{m \log m}$ (maybe after passing to a subsequence to make $\lim \frac{\log m}{\log \gamma_m}$ exist).  For sequences that grow slower than $\sqrt{m \log m}$ the rate function for large deviations is not $I(\gamma_m) - \log m$ anymore \cite{mikosch_ld}.
\end{rem}

\section{Discussion of the sharpness of Theorem \ref{thm:concentration-general}}

In this section, we would like to discuss that the bounds offered by Theorem \ref{thm:concentration-general} are sharp if compared with the limiting expressions obtained from the large deviation results. We clarify this point through the following two examples: 
Let $I(t)$ capture the right tail of a centered random variable $X$ and also captures its right tail in the limit ($I_{br}(t)$ has this property).  Assume that $X_1, ..., X_m$ are independent copies of $X$ and $S_m = \sum\limits_{i \leq m} X_i$.  Below we discuss the subWeibull distributions and the distributions with polynomial tail decays.

\begin{enumerate}
\item $I(t) = a_{\alpha} \sqrt[\alpha]{t}$: Theorem \ref{thm:concentration-general} yields

\begin{equation} \label{eq:dicussion-subweibull}
\p{S_m > \gamma_m} \leq 
\begin{dcases}
\expp{-c_{\frac{\gamma_m}{m}} \beta  I(\gamma_m)} + m \expp{- I(\gamma_m)},
&
\gamma_m \gg {m}^{\frac{\alpha }{2 \alpha - 1}},
\\
\expp{- \frac{{ \gamma_m}^2}{2 m \intoo{\sigma^2 + \epsilon}} } + m \expp{- \frac{m t_{\max} ^2}{\beta( \sigma^2 + \epsilon )}},
&
\gamma_m \ll {m}^{\frac{\alpha}{2 \alpha - 1}}. 
\end{dcases}
\end{equation}
Note that $1 - \beta$ and $\epsilon$ can be chosen arbitrarily small.  Moreover,  $\lim\limits_{\gamma_m \to \infty} c_{\frac{\gamma_m}{m}} = 1$.  Hence, in the first case the right hand side behaves like its dominant term which is $\expp{- I(\gamma_m)}$.  As proven in Theorem \ref{thm:ld-general}, $\p{S_m > \gamma_m} \sim \expp{- I(\gamma_m)}$ which proves the asymptotic sharpness of our first bound. 
Furthermore, when $\gamma_m \ll {m}^{\frac{\alpha}{2 \alpha - 1}}$ the right hand side of Inequality \eqref{eq:dicussion-subweibull} behaves like $\expp{-\frac{\gamma_m^2}{2 m \sigma^2}}$.  It is known for $\gamma_m$ growing at this speed we have
\begin{equation*}
\lim_{m \to \infty} \frac{\p{S_m > \gamma_m}}{\bar{\Phi} \intoo{ \frac{\gamma_m}{\sigma \sqrt{m} }}} = 1,
\end{equation*}
where $\bar{\Phi} = 1 - \Phi(t)$ and $\Phi$ is the CDF of standard normal distribution \cite{mikosch_ld}.  Since ${\bar{\Phi} \intoo{ \frac{\gamma_m}{\sigma \sqrt{m} }}}\sim \frac{ \sqrt{m} \sigma}{\sqrt{2 \pi} \gamma_m} \expp{ - \frac{\gamma_m^2}{2m \sigma^2}}$ we have
\begin{equation*}
\lim_{m \to \infty} \frac{- \log \bar{\Phi} \intoo{\frac{\gamma_m}{\sigma \sqrt{m}}}}{ \frac{\gamma_m^2}{2 m \sigma^2}} = 1.
\end{equation*}
Hence,
\begin{equation*}
\lim_{m \to \infty} \frac{- \log \p{S_m > \gamma_m}}{ \frac{\gamma_m^2}{2 m \sigma^2}} = 1.
\end{equation*}

This proves the asymptotic sharpness of our second bound.

\item $I(t) = \gamma \log t$ for $\gamma > 2$: Theorem \ref{thm:concentration-general} yields

\begin{equation*}
\p{S_m > \gamma_m} \leq 
\begin{dcases}
\expp{-c_{\frac{\gamma_m}{m}} \beta  I(\gamma_m)} +  \expp{- \intoo{I(\gamma_m) - \log m}},
&
\gamma_m \gg \sqrt{m \log m},
\\
\expp{- \frac{{ \gamma_m}^2}{2 m \intoo{\sigma^2 + \epsilon}} } + m \expp{- \frac{m t_{\max} ^2}{\beta( \sigma^2 + \epsilon )}},
&
\gamma_m \ll \sqrt{m \log m}, 
\end{dcases}
\end{equation*}
for any $\beta < 1 - \frac{2}{\gamma}$.  This time, one can easily check $m \expp{- I(\gamma_m)} = m \gamma_m^{- \gamma}$ and $\expp{- \frac{{ \gamma_m}^2}{2 m \intoo{\sigma^2 + \epsilon}} }$ will be dominant for the first and second cases, respectively.   One more time, the rate function given by Corollary \ref{coro:cbeta-for-poly-tail} in the first case, and the Gaussian CDF approximation in the second case \cite{mikosch_ld} match the dominant terms offered by Theorem \ref{thm:concentration-general}.  

\end{enumerate}

\riii{
\section{Comparison of Theorem \ref{thm:concentration-general} with existing inequalities}\label{sec:compall}

This Section aims to present a detailed comparison between the main result of this paper and several existing inequalities in the literature.  We choose the inequalities that bear the most similarities to our result, in terms of both assumptions and objective. For each existing inequality, first we mention the original result, with a slight adaptation of the notations, and then compare it with our Theorem \ref{thm:concentration-general}. As will be clarified Theorem \ref{thm:concentration-general} is better than the existing results in that the same inequality works on a large class of distributions and is always asymptotically sharp.

In what follows, we start with some of the existing results mentioned in the work of Nagaev \cite{nagaev1979large} followed by more recent results of \cite{adamczak2008tail, rio2017constants}.

\subsection{Theorem 1.3 and Corollary 1.8 of Nagaev (1979) \cite{nagaev1979large}}

In this section, we would like to review one of the main results of \cite{nagaev1979large}.

For $y_1, y_2, \ldots, y_n \in \mathbb{R}$, $t \geq 2$, $0< \alpha<1$, and $\beta= 1- \alpha$, define $\mvec{Y} = (y_1, ..., y_n)$, and 
\begin{align*}
    &
    A(t; 0, \mvec{Y}) = \sum_1^m \int_{0 \leq u \leq y_i} \abs{u}^t dF_i(u)
    , \quad
    \mu(-\infty, \mvec{Y}) = \sum_1^m \int_{ u \leq y_i} u dF_i(u)
    , \quad
    B^2(-\infty, \mvec{Y}) = \sum_1^m \int_{ u \leq y_i} u^2 dF_i(u)
    \\ &
    P_4 = \expp{ \beta \frac{x}{y} - \intoo{\intoo{1 - \frac{\alpha}{2}} \frac{x}{y} - \frac{\mu(-\infty, \mvec{Y})}{y}}. \log \intoo{\frac{\beta x y^{t - 1}}{A(t; 0, \mvec{Y})} + 1}  }
    \\ &
    P_5 = \expp{\intoo{\beta - \frac{t \alpha}{2} }\frac{x}{y} - \intoo{\beta \frac{x}{y} - \frac{\mu(-\infty, \mvec{Y})}{y}}. \log \intoo{\frac{\beta x y^{t - 1}}{A(t; 0, \mvec{Y})} + 1}}
    \\ &
    P_6 = \expp{- \frac{\alpha x \intoo{\frac{\alpha x}{2} - \mu(-\infty, \mvec{Y})}}{\rme^t B^2(-\infty, \mvec{Y})} }.
\end{align*}
One of the main results of \cite{nagaev1979large} is the following concentration result:

\begin{thm*}[1.3 from \cite{nagaev1979large}]
\label{thm:nagaev 1.3}
Suppose $t \geq 2, \; 0 < \alpha < 1$, and $\beta = 1 - \alpha$. 
If
\begin{equation}
    \label{eq:nagaev asum 1}
    \max \left[ t, \log \intoo{\frac{\beta x y^{t - 1}}{A(t; 0, \mvec{Y})} + 1} \right] \geq \frac{\alpha x y}{\rme^t B^2(-\infty, \mvec{Y})}
\end{equation}
then
\begin{equation}
    \label{eq:nagaev upbdd 1}
    \p{S_m \geq x} \leq \sum\limits_1^m \p{X_i > y_i} + P_6.
\end{equation}
If
\begin{equation}
    \label{eq:nagaev asum 2}
    \max \left[ t, \log \intoo{\frac{\beta x y^{t - 1}}{A(t; 0, \mvec{Y})} + 1} \right] < \frac{\alpha x y}{\rme^t B^2(-\infty, \mvec{Y})}
\end{equation}
then
\begin{equation}
    \label{eq:nagaev upbdd 2}
    \p{S_m \geq x} \leq \sum\limits_1^m \p{X_i > y_i} + P_4.
\end{equation}
If instead of \eqref{eq:nagaev asum 2} at least one of the conditions $\beta \geq t \frac{\alpha}{2}$ and $\beta x \geq \mu(-\infty, \mvec{Y})$ is fulfilled, then
\begin{equation}
    \label{eq:nagaev upbdd 3}
    \p{S_m \geq x} \leq \sum\limits_1^m \p{X_i > y_i} + P_5.
\end{equation}

\end{thm*}
In the above Theorem, it is assumed that $\e{\abs{X_i}^t} < \infty$ for some $t \geq 2$ and  $y \geq \max \cbr{y_1, ..., y_n}$.

As is clear, the main advantage of the above theorem is that it covers the large class of distributions with bounded $t$-moment. For instance, it does not require the assumption of identical distributions. However, as we discussed in the introduction, it leaves several free parameters for the user to choose, e.g. $(y_1, ..., y_n), y, \alpha$.  It is not clear what choice of these parameters lead to an optimal upper bound.  To obtain an optimal upper bound one requires to do the following two tedious tasks:
\begin{enumerate}
    \item Find sharp bounds for $P_4, P_5, P_6$ for arbitrary values of $\alpha$, $\beta$, $y$.
    \item Then, the user needs to plug in this bounds in \eqref{eq:nagaev upbdd 1} or \eqref{eq:nagaev upbdd 2} and find the optimal choice of $\alpha$, $\beta$, $y$ that generate the best upper bound. 
\end{enumerate}

 As is clear, tuning all of these free parameters optimally is a challenging task. In comparison, the user of Theorem \ref{thm:concentration-general} of this manuscript, does not need to tune any particular parameter. The only task that is left to the user is to find an upper bound of $\nu(mt, \beta)$, for which we have provided a few simple integral formulas.   Furthermore, Theorem \ref{thm:concentration-general} clearly separates the region of small deviations, where the tail looks like Gaussian, and the region of large deviations, where the tail behavior is dictated by the rate function $I(\cdot)$. However, it is not straightforward to see the double behavior of the tail bound and the boundary of its change of behavior from equations \eqref{eq:nagaev upbdd 1}, \eqref{eq:nagaev upbdd 2}, and \eqref{eq:nagaev upbdd 3}.  As discussed in Section \ref{sec:large deviation}, Theorem \ref{thm:concentration-general} not only determines the boundary of the two regions of deviation, but also offers an upper bound that is sharp asymptotically.\footnote{It is worth to mention that in \cite{nagaev1979large} there are also a couple of Theorems before the aforementioned one that cover $0 < t < 2$. These classes of Distributions are not covered by our results, because their variances are infinite.} In an effort to overcome the issues that we raised above, \cite{nagaev1979large} obtains the following corollary of the above theorem. 

\begin{coro*}[1.8 from \cite{nagaev1979large}]
If $\e{X_i} = 0$ and $A_t^+ < \infty, \; t \geq 2$, then
\begin{equation}
    \label{eq:nagaev coro upbd}
    \p{S_m \geq x} \leq c_t^{(1)} A_t^+ x^{-t} + \expp{-c_t^{(2)}\frac{x^2}{B_m^2}},
\end{equation}
where $c_t^{(1)} = \intoo{1 + \frac{2}{t}}^t$, $c_t^{(2)}  = 2 (t + 2)^{-2} \rme^{-t}$, and $A_t^+ = \sum\limits_1^n \int_{u \geq 0} u^t dF_i(u)$.
\end{coro*}

To clarify the difference of this result with ours, we assume $X_1, ..., X_m$ are iid and consider the following values for the remaining parameters in \eqref{eq:nagaev coro upbd}:
$t = 2, \; x = mu, \;   A_2^+ = m \sigma_+^2 , \; B_m^2 = m \sigma^2$.  Then \eqref{eq:nagaev coro upbd} reads:
\begin{equation}
    \label{eq:nagaev coro special case}
    \p{S_m \geq mu} \leq \frac{9 \sigma_+^2}{m u^2} + \expp{- \frac{m u^2}{8 \rme^2 \sigma^2}}
\end{equation}

While the double behavior of the deviation is apparent in this corollary, it is not asymptotically sharp for either small or large deviation regimes: 
\begin{itemize}
\item For $ u = O \intoo{ \frac{1}{\sqrt{m}}}$, the upper bound of \eqref{eq:nagaev coro special case} behaves as $\expp{- \frac{m u^2}{8 \rme^2 \sigma^2}}$ compared to the upper bound offered by our result, i.e. \eqref{eq:concentration one regime}: $\expp{-\frac{m u^2}{2 v(mu, \beta)}}$.  Given $v(mu, \beta) \xrightarrow{mu \to \infty} \sigma^2$ by Lemma \ref{lem:limit-of-c}, the asymptotic behavior of $\expp{-\frac{m u^2}{2 v(mu, \beta)}}$ and $\expp{-\frac{m u^2}{2 \sigma^2}}$ are the same. Hence, our result matches the tail of limiting distribution given by the central limit theorem.  Hence, the exponent of \eqref{eq:nagaev coro special case} is off by a factor $8 \rme^2 \simeq 59.11$ in this regime.  
\item When $ u \gg \frac{1}{\sqrt{m}}$, the dominant term of \eqref{eq:nagaev coro special case} is $\frac{9 \sigma_+^2}{m u^2}$, while the dominant term of our bound given in Remark \ref{rmk:one regime ineq} is $\expp{-\beta \max \cbr{c_u, \frac{1}{2}} I(m u)} + m \expp{-I(mu)}$.  Given the fact that $c_u \to 1 $, and $\beta$ can be chosen  close to $1$ as $mu \to \infty$, we can asymptotically capture tail behavior of $(m + 1) \expp{-I(mu)}$ which has been shown to be sharp in Section \ref{sec:large deviation}.  Note that for having finite second moment we need $I(u) > 2 log(u)$, hence $(m + 1) \exp(-I(m u)) \ll \frac{9 \sigma_+^2}{m u^2}$ which again shows the upper bound of \eqref{eq:nagaev coro special case} is not sharp in this regime of deviation. 
\end{itemize}

Note that the fact that Corollary 1.8 of \cite{nagaev1979large} confirms our main claim in the introduction that finding sharp bounds for the quantities involved in Theorem 1.3 of \cite{nagaev1979large} and tuning the parameters optimally is a challenging task by itself. 

\subsection{Theorem 4 of Adamczak,  R. (2008) \cite{adamczak2008tail}} 

In this section, we compare our main result, with the main result of \cite{adamczak2008tail}. Let $\psi_\alpha (x) = {\rm e}^{x^\alpha} -1$. For a random variable $X$, define also the Orlicz norm:
\[
\|X\|_{\psi_\alpha} = \inf\{\lambda>0 : \mathbb{E} \psi_\alpha (|X|/\lambda) \leq 1\}.
\]
We then have:

\begin{thm*}[4 from \cite{adamczak2008tail}]
Let $X_1, ..., X_n$ be independent random variables with values in a measurable space $(\mathcal{S, B})$ and $\mathcal{F}$ be a countable class of measurable functions $f: \mathcal{S} \to \mathbb{R}$.  Assume that for every $f \in \F$ and every $i, \e{f(X_i)} = 0$ and for some $\Tilde{\alpha} \in (0, 1]$ and all $i, \norm{\sup_f \abs{f(X_i)}}_{\psi_{\Tilde{\alpha}}} < \infty$.  Let
\begin{equation*}
    Z = \sup_{f \in \F} \abs{\sum_{i = 1}^n f(X_i)}.
\end{equation*}
Define moreover
\begin{equation*}
    \sigma^2 = \sup_{f \in \F} \sum_{i = 1}^n \e{f(X_i)^2}.
\end{equation*}
Then, for all $0 < \eta < 1$ and $\delta > 0$, there exists a constant $C = C(\Tilde{\alpha}, \eta, \delta)$, such that for all $t \geq 0$,

\begin{align*}
    \mathbb{P} \left( Z  \right. & \left. \geq (1 + \eta) \e{Z} + t \right)
    \\ & \leq
    \expp{- \frac{t^2}{2(1 + \delta) \sigma^2}} + 3 \expp{- \intoo{\frac{t}{C \norm{\max_i \sup_{f \in \F} \abs{f(X_i)}_{\psi_{\Tilde{\alpha}}}}}}^{\Tilde{\alpha}}
    }.
\end{align*}

and
\begin{align*}
    \mathbb{P} \left( Z  \right. & \left. \leq (1 - \eta) \e{Z} - t \right)
    \\ & \leq
    \expp{- \frac{t^2}{2(1 + \delta) \sigma^2}} + 3 \expp{- \intoo{\frac{t}{C \norm{\max_i \sup_{f \in \F} \abs{f(X_i)}}_{\psi_{\Tilde{\alpha}}}}}^{\Tilde{\alpha}}}.
\end{align*}
\end{thm*}

 The main objective of this theorem is to obtain a maximal inequality for a class of random processes with a common source of randomness.  Nevertheless, if the class $\F$ is a singleton that only consists the identity function and the random variables $X_1, ..., X_n \sim X$ have the same distribution, then the tail bound given here is similar to what we pursued in the current article.  Again for notational simplicity we consider the case of centered random variables.  If $\e{X_i} =0$, then Theorem 4 of \cite{adamczak2008tail} is simplified to:
 
\begin{equation}
    \label{eq:adamczak tail bound}
    \p{S_m \geq m t} \leq  \expp{- \frac{m t^2}{2(1 + \delta) \Tilde{\sigma}^2}} + 3 \expp{-\intoo{\frac{m t}{C \norm{X}_{\psi_{\Tilde{\alpha}}}}}^{\Tilde{\alpha}}},
\end{equation}

where $\Tilde{\sigma} = \e{X^2}$.

Now this inequality can be directly compared with our results:
\begin{enumerate}
\item First, as is clear, \eqref{eq:adamczak tail bound} requires $\norm{X}_{\psi_{\Tilde{\alpha}}} < \infty$ for some $\Tilde{\alpha} \leq 1$. Hence, it can only be applied to the subWeibull distributions for which the rate function satisfies $I(t) \geq c t^{\Tilde{\alpha}}$. Hence, our result handles more general heavy tails. 
\item Within the class of distributions with finite $\norm{X}_{\psi_{\Tilde{\alpha}}}$, \eqref{eq:adamczak tail bound} sharply characterizes the deviation in the Gaussian regime. However, in most cases, it is not easy to extract the exact constants of the rate function for deviations with larger sizes, for the following two reasons:  
\begin{itemize}
\item First, usually it is not straightforward to compute the exact value of $\norm{X}_{\psi_{\Tilde{\alpha}}}$ for a given distribution.  
\item More importantly, this theorem does not specify the exact value of $C=C(\Tilde{\alpha}, \eta, \delta)$.  Taking a closer look at the proof steps of this theorem also reveals that this work was not dedicated to achieve the best possible value for this constant. In particular, the proof in \cite{adamczak2008tail} writes: ``Without loss of generality we may and will assume that
\begin{equation*}\label{eq:condition_adamzek}
t / \norm{ \max_{1 \leq i \leq n} \sup_{f \in \mathcal{F}} |f(X_i)}_{\psi_{\alpha}} > K (\alpha, \eta, \delta)
\end{equation*}
otherwise we can make the theorem trivial by choosing the constant $C = C(\eta, \delta, \alpha)$ to be large enough." The value of $K (\alpha, \eta, \delta)$ will also be determined later based on the requirements of the proof.  
\end{itemize}
\end{enumerate}
This causes several major issues for a user of Theorem 4 of \cite{adamczak2008tail}. 
\begin{enumerate}
    \item The user of the  theorem should find the value of $C$ and $\norm{X}_{\psi_{\Tilde{\alpha}}}$, both of which are not straightforward. 
    \item Even if the user can find the best value of $C$, given the condition \eqref{eq:condition_adamzek} the concentration presented in the paper is not accurate for certain values of $t$ and it is difficult to figure out the range on which the inequalities are accurate.
    \item Even if one can calculate the exact values of $\norm{X}_{\psi_{\Tilde{\alpha}}}$ and $C$, it is not clear that the theorem offers a sharp rate function as we did in Section \ref{sec:large deviation}.

\end{enumerate}
   In comparison, note that Theorem \ref{thm:concentration-general} and Remark \ref{rmk:one regime ineq} obtain precise constants, that are asymptotically sharp.  Moreover, they cover a larger class of distributions (distributions with finite second moments) in a single inequality.  

\subsection{Corollary 4.2 of Rio, E. (2017) \cite{rio2017constants}}

\begin{coro*}[4.2 of \cite{rio2017constants}]

Assume $(M_j)_{0 \leq j \leq m}$ is a martingale in $L^2$ with respect to a non-decreasing filtration $(\F_j)$, such that $M_0 = 0$.  Set $X_j = M_j - M_{j - 1}$.  We assume that, for some constant $r > 2$,
\begin{equation}
    \label{eq:rio assumption 1}
    \norm{\e{X_j^2 \sVert \F_{j- 1}}} < \infty,
    \quad \text{ and } \quad
    \norm{\sup_{t > 0} \intoo{t^r \p{X_{j+} > t} \sVert \F_{j - 1}}}_\infty < \infty,
\end{equation}
for any $j \in \intcc{1, m}$.  We set
\begin{equation}
    \label{eq:rio parameter definition}
    \Tilde{\sigma}^2 = \norm{\sum_{j = 1}^m \e{X_j^2 \sVert \F_{j - 1}} },
    \quad \text{ and } \quad
    C_r^w(M) = \norm{\sup_{t > 0} \intoo{t^r \sum_{j = 1}^m \p{X_{j+} > t \sVert \F_{j - 1}}}}_\infty^{1/r},
\end{equation}
where $X_{j+} = \max \intoo{0, X_j}$.  Then, for any $z > 1$,
\begin{equation}
    \label{eq:rio upbdd by coro 4.2}
    \p{\max(M_0, M_1, ..., M_m) > \Tilde{\sigma} \sqrt{2 \log z} + C_r^w(M) \mu_r z^{1/r}} \leq 1/z,
\end{equation}
where $\mu_r = 2 + \max(4/3, r/3)$.
\end{coro*}

Clearly, one advantage of this result over ours is that it can be applied to dependent variables as long as their sum forms a Martingale. However, given that the sum of dependent variables is not the focus of the current paper, let us focus on $X_1, ..., X_m \overset{iid}{\sim} X$ with $\e{X} = 0$ and $\e{X} = \sigma^2$. In this case, \eqref{eq:rio upbdd by coro 4.2} is simplified to:
\begin{equation}
    \label{eq:rio in iis case}
    \p{S_m >  \sigma \sqrt{2 m \log z} + C_r^w(M) \mu_r z^{1/r}} \leq 1/z,
\end{equation}
where $\sigma^2 = \frac{\Tilde{\sigma}^2}{m} = \e{X^2}$. First, note that the structure of this concentration result is particularly tailored to the distributions for which $P(X > t)$ has a polynomial decay. For instance, the concentration is not particularly useful for subWeibull distributions.
So, the concentration results presented in this paper can be applied to more general distributions. Now, let us ignore this limitation and compare \eqref{eq:rio in iis case} with our own results on the types of distributions that \eqref{eq:rio in iis case} is designed for. 

Similar to \eqref{eq:adamczak tail bound}, \eqref{eq:rio in iis case} is sharp in the Gaussian regime, where $ \sigma \sqrt{2 m \log z}$ is the dominant term.  In this regime, if we ignore the second term in the upper bound we achieve
\begin{equation*}
    \p{S_m > \sigma \sqrt{2 m \log z}} \leq 1/z,
\end{equation*}
which is equivalent to
\begin{equation*}
    \p{S_m > m t} \leq \expp{- \frac{t^2 m}{2 \sigma^2}}.
\end{equation*}
When $z$ is large enough to make the second term of \eqref{eq:rio in iis case} dominant, by ignoring the first term we can obtain
\begin{equation}
    \label{eq:rio iid LD regim}
    \p{S_m > C_r^w(M) \mu_r z^{1/r}} \leq 1/z,
\end{equation}
where $C_r^w(M) =  \abs{m \sup\limits_{t > 0} t^r \p{X > t} }^{1/r} $. Suppose that we are interested in a distribution that satisfies $\p{X > t} \sim \frac{1}{t^r}$. Then, we  have  $t^r \p{X > t} = \Theta(1)$, and as $m \to \infty, \quad C_r^w(M)$  grows like $m^{1/r}$. Hence, the upper bound given by \eqref{eq:rio iid LD regim} is equivalent to
\begin{equation}
    \label{eq:rio ld simplified}
    \p{S_m > m t} \leq \intoo{C_r^w(m) \mu_r}^r \frac{1}{(m t)^r} = \frac{m \mu_r^r}{(m t)^r}.
\end{equation}
The upper bound given by Theorem \ref{thm:concentration-general} for large $t$, (or large $z$ in Rio's notation), is
\begin{equation}
    \label{eq:our bound in compar to rio}
    \p{S_m > m t} \leq \expp{-c_t \beta I(m t)} + m \expp{-I(mt)} = \frac{1}{(mt)^{c_t \beta r}} + \frac{m}{(m t)^r}.
\end{equation}

Hence, as is clear, \eqref{eq:rio ld simplified} has a slightly tighter order since $c_t \beta < 1$ in \eqref{eq:our bound in compar to rio}.  Nevertheless both bounds are tight enough to capture the asymptotic decay of $\p{S_m > mt}$ and recover its logarithmic rate function (see Section \ref{sec:large deviation}). However, note that Rio's bound is only sharp for very specific types of distributions with certain tail behavior, while our result offers sharp results on a wide range of distributions.

}

\section{Proofs of our main results} \label{sec:proofs}

In this section, we state and prove a key lemma about the truncated random variable. This lemma is important in the proof of our concentration and large deviation results.

\subsection{Proof of Theorem \ref{thm:concentration-general}}

\begin{proof}[Proof of Theorem \ref{thm:concentration-general}]
Note that by Lemma \ref{lem:mgf-bound-general} and \eqref{eq:bdd by c assumption} we have 
\begin{equation*} \label{eq:mgf-bound-concentration-general}
\log \e{\expp{\lambda (X^L - \e{X})}} \leq \frac{v \intoo{L,\beta}}{2} \lambda^2.
\end{equation*}
Moreover,
\begin{align}
\p{S_m - \e{S_m} > mt} & \leq
\p{\sum X_i^L - \e{S_m} > mt} + \p{\exists i \quad X_i > L} \nonumber
\\ & \overset{(\star)}{\leq}
\expp{- \lambda m t} \e{\expp{\lambda (X^L - \e{X} ) }}^m + m \p{X > L}
\nonumber \\ & \leq \label{ineq:deviation-bound-general}
\expp{m \intoo{- \lambda t + \frac{v \intoo{L,\beta}}{2} \lambda^2}} + m \expp{-I(L)}.
\end{align}

\rii{To obtain the inequality marked by $(\star)$ we used Markov's inequality.} The main remaining step is to find good choices for the free parameters $L$ and $\lambda$. The goal is to choose the values of $\lambda, L$ such that we get the best upper bound in \eqref{ineq:deviation-bound-general}. We consider two cases: (i) $t > t_{\max}$, and (ii) $t \leq t_{\max}$. In each case, we select these parameters accordingly. 

\begin{itemize}

\item Case 1 ($t > t_{\max}$): In this case, we choose $L = m t$ and $\lambda = \beta \frac{I(mt)}{mt}$. We have
\begin{align*}
\p{S_m - \e{S_m} > mt} &\leq
 \expp{ - \beta \intoo{1 - \frac{\beta v \intoo{L,\beta} I(mt)}{2 m t^2}}  I(mt) } + m\expp{- I(mt)}
\\  & = 
\expp{ - \beta c_t I(mt) } + m\expp{- I(mt)}.
\end{align*}

Note that since for all $t > t_{\max}$ we have $t > \beta v \intoo{L,\beta} \frac{I(mt)}{mt}$,  we can conclude $\frac{1}{2} \leq c_t < 1.$ \\

\item Case 2 ($t \leq t_{\max}$): In this case, we pick $L = m t_{\max}$ and $\lambda = \frac{t}{v \intoo{L,\beta}} \leq \frac{t_{\max}}{v \intoo{L,\beta}} = \beta \frac{I(L)}{L}$. Then, \eqref{ineq:deviation-bound-general} implies
\begin{align}
\p{S_m - \e{S_m} > mt} &\leq
\expp{-\frac{1}{2 v \intoo{L,\beta}} m t^2} + m \expp{- I(m t_{\max})} \nonumber
\\ & =
\expp{-\frac{1}{2 v \intoo{L,\beta}} m t^2} + m \expp{- \frac{1}{\beta v \intoo{L,\beta}} m t_{\max}^2}. \nonumber 
\end{align}

Note that $v \intoo{L, \beta}$ is increasing in $\beta$.  Hence, choosing a smaller value for $\lambda$, as we did in this case, causes no problem.

\end{itemize}

\end{proof}

\subsection{Proof of Remark \ref{rmk:one regime ineq}}
\label{sec:proof:remk one regime}

\begin{proof}
Let $L = mt$. Using \eqref{ineq:deviation-bound-general} and the fact that $v(L, \beta)$ is an increasing function of $\beta$ we obtain:
\begin{equation}
    \label{eq:proof:one regim general}
    \p{S_m > mt} \leq \expp{m \intoo{- \lambda t + \frac{v \intoo{mt,\beta}}{2} \lambda^2}} + m \expp{-I(mt)},
    \quad \forall \lambda \leq \frac{\beta I(mt)}{mt}
\end{equation}
To achieve \eqref{eq:concentration one regime} we need to show that we always have a choice for $\lambda \leq \frac{\beta I(mt)}{mt}$ such that
\begin{equation*}
    \expp{m \intoo{- \lambda t + \frac{v \intoo{mt,\beta}}{2} \lambda^2}} \leq  \expp{-\frac{m t^2}{2 v(mt, \beta)}} + \expp{-\beta \max \cbr{c_t, \frac{1}{2}} I(m t)}.
\end{equation*}


We consider two cases:
\begin{itemize}
    \item Case 1 ($c_t \geq \frac{1}{2}$):
    Choose $\lambda = \frac{\beta I(mt)}{mt}$ and we get
    $\expp{m \intoo{- \lambda t + \frac{v \intoo{mt,\beta}}{2} \lambda^2}} \leq \expp{- \beta c_t I(mt)} = \expp{- \beta \max \cbr{c_t, \frac{1}{2}} I(mt)}$
    
    \item Case 2 ($c_t < \frac{1}{2}$):
    Having $c_t < \frac{1}{2}$ means $\frac{t}{v(mt, \beta)} < \frac{\beta I(mt)}{mt}$, so we may choose $\lambda = \frac{t}{v(mt, \beta)}$ that leads to 
    \begin{equation*}
        \expp{m \intoo{- \lambda t + \frac{v \intoo{mt,\beta}}{2} \lambda^2}} =  \expp{-\frac{m t^2}{2 v(mt, \beta)}}.
    \end{equation*}
\end{itemize}

\end{proof}

\subsection{Proof of Lemma \ref{lem:limit-of-c}}\label{ssec:proof:lem1}
First, we prove the lemma under Assumption (a).
Note that for $L > \e{X}$ we have $ \intoo{X^L - \e{X}}^2 \leq   \intoo{X - \e{X}}^2 \in \mathcal{L}^1$. Furthermore, $X^L \xrightarrow{a.s.} X$. Hence, by using the dominant convergence theorem we obtain
\begin{equation}
\e{ \intoo{X^L - \e{X}}^2 \mathds{I} \intoo{X \leq \e{X}} } \xrightarrow{L \to \infty} \e{(X - \e{X})^2 \mathds{I} \intoo{X \leq \e{X}}} . 
\end{equation}

Furthermore,  \rii{noting $0 \leq \lambda_{L, \beta} \leq \beta \frac{I(L)}{L} \rightarrow 0$ as $L \rightarrow \infty$} it is straightforward to show that 
\begin{equation}
(X^L - \e{X})^2 \expp{\lambda_{L, \beta} (X^L - \e{X}} \xrightarrow{a.s.} (X - \e{X})^2. 
\end{equation}
Hence, if we find an $\mathcal{L}^1$ function that dominates $(X^L - \e{X})^2 \expp{\lambda_{L, \beta} (X^L - \e{X}}$, then we can use the dominant convergence theorem to complete the proof.  Toward this goal, we consider $$Y =  \intoo{X - \e{X}}^2 \expp{\beta c_\alpha \sqrt[\alpha]{\max(X, 0)} + 1} \mathds{I} \intoo{X > \e{X}}.$$
Note that for $X > \e{X}$, $L > 2 \e{X} $ and $- \lambda_{L, \beta} \e{X} \leq 1 $, we have 
\begin{equation}
\expp{\lambda_{L, \beta} (X^L - \e{X})}  \leq \expp{\lambda_{L, \beta} X^L + 1} \leq \expp{ \beta  c_\alpha \frac{\sqrt[\alpha]{L}}{L} X^L + 1} \leq \expp{\beta c_\alpha \sqrt[\alpha]{\max(X, 0)} + 1}.
\end{equation}
Thus, for $L$ large enough we have
\begin{equation}
(X^L - \e{X})^2 \expp{\lambda_{L, \beta} (X^L - \e{X})} \mathds{I} \intoo{X > \e{X}} \leq Y.
\end{equation}
 To prove the integrability of $Y$, note that 
\begin{align}
\mathbb{E} & \left[  \intoo{X - \e{X}}^2  \expp{ \beta c_\alpha \sqrt[\alpha]{\max (X, 0)}} \mathds{I}(X > \e{X}) \right] \nonumber 
\\ & =
\int_0^\infty \p{\intoo{X - \e{X}}^2 \expp{ \beta c_\alpha \sqrt[\alpha]{\max (X, 0)}} > u, X > \e{X}} du \nonumber
\\ & \leq
\e{\intoo{X - \e{X}}^2 \mathds{I} \intoo{\e{X} \leq X < 0}} + 
\int_0^\infty \p{X > t} du \qquad (t - \e{X})^2 \expp{\beta c_\alpha \sqrt[\alpha]{t}} = u \nonumber
\\ & \leq
Var(X) + \int_0^\infty \expp{- c_\alpha \sqrt[\alpha]{t}} du \nonumber 
\\ & \leq
Var(X) + \int_0^\infty \expp{- c_\alpha \sqrt[\alpha]{t}} \intoo{2(t - \e{X}) + \frac{\beta c_\alpha}{\alpha} t^{\frac{1}{\alpha} - 1} (t - \e{X})^2} \expp{\beta c_\alpha \sqrt[\alpha]{t	}} dt \nonumber
\\ & \leq
Var(X) + \int_0^\infty \expp{- c_\alpha(1 - \beta) \sqrt[\alpha]{t}} {\rm Poly} \intoo{t^{\frac{1}{\alpha} - 1}, t} dt  < \infty. \nonumber
\end{align}
Recall that $\beta < 1$ and $c_\alpha > 0$, hence the exponent of the last line is negative.  Thus $Y$ is integrable as it was desired.

The proof under assumption (b) is analogous to the proof of part (a). The only difference is to prove the dominant convergence theorem for the following variable:
\[
(X^L - \e{X})^2 \expp{\lambda_{L, \beta} (X^L - \e{X)}} \mathds{I} \intoo{X > \e{X}}.
\]
Toward this goal we use the dominant variable:
\begin{align*}
Y & =  \intoo{X - \e{X}}^2 \expp{\beta \gamma \log \intoo{X - \e{X}} } \mathds{I} \intoo{X > \e{X}} 
\\ & =
\intoo{X - \e{X}}^{2 + \beta \gamma}  \mathds{I} \intoo{X > \e{X}}.
\end{align*}
The proof of the integrability of this variable is left to the readers.

\subsection{Proof of Theorem \ref{thm:ld-general}} \label{sec:proof:theorem2}

We start with a lemma that will be used in our proof later. 
\begin{lem} \label{lem:limit-of-log-of-sum}
Let ${a_n}, {b_n}$ and ${c_n}$ be sequences of positive numbers such that
\begin{equation*}
\lim_{n \to \infty} \frac {\log a_n}{c_n} = a, \quad \lim_{n \to \infty} \frac{\log b_n}{c_n} = b, \quad \lim_{n \to \infty} c_n = \infty.
\end{equation*}
Then
\begin{equation}
\lim_{n \to \infty} \frac{ \log (a_n + b_n) }{c_n} = \max \cbr{a, b}.
\end{equation}
\end{lem}

\vspace{.5 cm}
\begin{proof}
Without loss of generality assume $a \geq b$, hence $a_n \geq b_n$ for large enough $n$.  Thus
\begin{align*}
a = \lim_{n \to \infty} \frac {\log a_n}{c_n} \leq \lim_{n \to \infty} \frac {\log (a_n + b_n)}{c_n} \leq \lim_{n \to \infty} \frac {\log 2 a_n}{c_n} = \lim_{n \to \infty} \frac {\log 2}{c_n} + \lim_{n \to \infty} \frac {\log a_n}{c_n} = a.
\end{align*}
Therefore
\begin{equation*}
\lim_{n \to \infty} \frac{ \log (a_n + b_n) }{c_n} = a.
\end{equation*}
\end{proof}

First note that
\begin{align*} \label{eq:deviation based on one sample}
 \p{S_m - \e{S_m} > \gamma_m} \geq \p{X > \gamma_m} \p{ S_{m - 1} - \e{S_{m - 1}} \geq \e{X}}.
\end{align*}

Since $\frac{S_{m - 1} - \e{S_{m - 1}}}{\sqrt{m - 1}} \xrightarrow{d} \mathcal{N}(0, {\rm Var}(X))$ and $\frac{\e{X}}{\sqrt{m - 1}} \to 0$ we have 
\begin{equation*}
\p{ S_{m - 1} - \e{S_{m - 1}} \geq \e{X}} \geq C > 0,
\end{equation*}
for a positive constant $C$ and large enough $m$.  Therefore,

\begin{align} \label{eq:ld-upper-bound-general}
\lim_{m \to \infty} \frac{ - \log \p{S_m - \e{S_m} > \gamma_m} }{I(\gamma_m)} \leq \lim_{m \to \infty} \frac{- \log \p{X > \gamma_m}}{I(\gamma_m)} + \frac{-\log C}{I(\gamma_m)} = 1.
\end{align}
To obtain the last equality we used the fact that since $\log m \ll I(\gamma_m)$ we have $I(\gamma_m) \rightarrow  \infty$ as $m \rightarrow \infty$. Hence, 
\begin{equation*}
\lim_{m \to \infty} \frac{- \log C}{I(\gamma_m)} = 0.
\end{equation*}
On the other hand,
\begin{align}\label{eq:proofthm2_first}
\p{S_m - \e{S_m} > \gamma_m} &\leq 
\expp{- \lambda \gamma_m} \e{\expp{\lambda (X^L - \e{X})}}^m + m \p{X > L} \nonumber
\\ & \leq
\expp{- \lambda \gamma_m} \expp{\frac{ k_{L, \lambda}}{2} \lambda^2 m} + m \p{X > L},
\end{align}
where we used Lemma \ref{lem:mgf-bound-general} to obtain the last inequality. 
Let $L = \gamma_m$ and $\lambda = \beta \frac{I(\gamma_m)}{\gamma_m}$.  Moreover, assume $v_\beta$ is the bound for $v \intoo{L, \beta}$ when $L$ is large enough. Then, \eqref{eq:proofthm2_first} implies that

\begin{align} \label{eq:p-lower-bound-for-LD}
\p{S_m - \e{S_m} > \gamma_m} &\leq 
\expp{ - \beta I(\gamma_m) + \frac{\beta^2 c_{\beta}}{2} \frac{m I(\gamma_m)^2}{\gamma_m^2}} + m \p{X > \gamma_m} .
\end{align}
In order to find a lower bound for $\lim \frac{- \log \p{S_m - \e{S_m} > \gamma_m}}{I(\gamma_m)}$, we use Lemma \ref{lem:limit-of-log-of-sum}. Hence, we need to bound each term of \eqref{eq:p-lower-bound-for-LD} separately.

\begin{equation} \label{eq:p-ubd-term1}
\lim_{m \to \infty} \frac{\beta I(\gamma_m) - \frac{\beta^2 c_{\beta}}{2} \frac{m I(\gamma_m)^2}{\gamma_m^2} }{I(\gamma_m)} = \beta + \frac{\beta^2 v_\beta}{2} \lim_{m \to \infty} \frac{- m I(\gamma_m)}{\gamma_m^2} = \beta,
\end{equation}
where we used $I(\gamma_m) = o(\frac{\gamma_m^2}{m})$ to obtain the last equality.  Moreover,
\begin{equation} \label{eq:p-ubd-term2}
\lim_{m \to \infty} \frac{- \log(m \p{X > \gamma_m})}{I(\gamma_m)} = 1.
\end{equation}
The last equality holds because $I$ captures the tail of $X$ asymptotically and grows faster than $\log(m)$.
Hence, using \eqref{eq:p-lower-bound-for-LD}, \eqref{eq:p-ubd-term1} and \eqref{eq:p-ubd-term2} we obtain
\begin{equation*} 
\lim_{m \to \infty} \frac{- \log \p{S_m - \e{S_m} > mt}}{I(mt)} \geq \beta, \qquad \forall \beta < 1,
\end{equation*}

which implies
\begin{equation} \label{eq:ld-lower-bound-general}
\lim_{m \to \infty} \frac{- \log \p{S_m - \e{S_m} > mt}}{I(mt)} \geq 1.
\end{equation}
By using \eqref{eq:ld-upper-bound-general} and \eqref{eq:ld-lower-bound-general} we obtain
\begin{equation*}
\lim_{m \to \infty} \frac{ - \log \p{S_m - \e{S_m} > m t} }{I(m t)} = 1,
\end{equation*}
which concludes the proof.

\subsection{Proof of Corollary \ref{coro:ld-poly-tail}}\label{ssec:proof:lastcor}

First, assume $\gamma_m$ satisfies (i).
Let $\beta < 1 - \frac{k}{\alpha}$, hence $(1 - \beta) \alpha = k' > k$.  According to Corollary \ref{coro:cl-upper-bound-poly-tail} for this $\beta$ and $L = \gamma_m$ we have
\begin{equation*}
v \intoo{\gamma_m, \beta} \leq C \gamma_m^{2 - (1 - \beta) \alpha} \log \gamma_m = C \gamma_m^{2 - k'} \log \gamma_m.
\end{equation*}
Therefore
\begin{align*}
\frac{\gamma_m}{m} \gg   \gamma_m^{2 - k'} \frac{\log m}{\gamma_m} \geq C' \beta v \intoo{\gamma_m, \beta} \frac{I(\gamma_m)}{\gamma_m} ,
\end{align*}
since we have $\lim \frac{\log m}{\log \gamma_m} = k < k'$.  Thus, for large enough $m$, when applying Theorem \ref{thm:concentration-general} with $t = \frac{\gamma_m}{m}$ and the chosen $\beta$ above we will be in the $t > t_{\max}$ regime. 

For the second case that $\gamma_m$ satisfies (ii), Lemma \ref{coro:cl-upper-bound-poly-tail} implies that for any $\beta < 1 - \frac{2}{\alpha}$, $v \intoo{\gamma_m, \beta}$ remains bounded. Hence we have
\begin{equation*}
\frac{\gamma_m}{m} \gg \beta v \intoo{\gamma_m , \beta} \frac{I(\gamma_m)}{\gamma_m} = O \intoo{\frac{\log m}{\gamma_m}} ,
\end{equation*}

which means we still are in the region $t > t_{\max}$. Hence,
\begin{equation} \label{proof:ld-poly:prob-upper-bound}
\p{S_m > \gamma_m} \leq \expp{-c_ {\frac{\gamma_m}{m}} \beta I(\gamma_m)} + m \expp{- I(\gamma_m)}.
\end{equation}
Note that 
$ c_{\frac{\gamma_m}{m}} = 1 - \frac{1}{2} \frac{\beta v \intoo{\gamma_m, \beta}}{\frac{\gamma_m}{m}} \frac{I(\gamma_m)}{\gamma_m} \xrightarrow{m \to \infty} 1 $, so we obtain

\begin{align} \label{proof:ld-poly:first-term-upper-bound}
\lim_{m \to \infty} \frac{c_{\frac{\gamma_m}{m}} \beta I(\gamma_m)}{I(\gamma_m) - \log m} =
\lim_{m \to \infty} \frac{\beta }{1 - \frac{\log m}{I(\gamma_m)}} =
\lim_{m \to \infty} \frac{\beta}{1 - \frac{\log m}{\alpha \log  {\gamma_m}}} = \frac{\beta}{1 - \frac{k}{\alpha}} = \frac{\alpha - k'}{\alpha - k}, \qquad \forall k' > k.
\end{align}
Moreover,
\begin{align} \label{proof:ld-poly:second-term-upper-bound}
\lim_{m \to \infty} \frac{I(\gamma_m) - \log m}{I(\gamma_m) - \log m} = 1.
\end{align}
By combining \eqref{proof:ld-poly:prob-upper-bound}, \eqref{proof:ld-poly:first-term-upper-bound} and \eqref{proof:ld-poly:second-term-upper-bound} we obtain
\begin{equation*}
\lim_{m \to \infty} \frac{- \log \p{S_m > \gamma_m}}{I(\gamma_m) - \log m} \geq \frac{\alpha - k'}{\alpha - k}, \qquad \forall k' > k,
\end{equation*}
which implies
\begin{equation} \label{proof:ld-poly:ld-lower-bound}
\lim_{m \to \infty} \frac{- \log \p{S_m > \gamma_m}}{I(\gamma_m) - \log m} \geq 1.
\end{equation}
On the other hand, 
\begin{align} \nonumber
\p{S_m > \gamma_m} &\geq
 \sum_{j = 1}^m \p{\sum_{i \neq j} X_i > - \epsilon \sqrt{m}, \quad \max\limits_{i \neq j} X_i < \gamma_m} \p{X_j \geq \gamma_m + \epsilon \sqrt{m}}
 \\ & = \nonumber
 m \p{\frac{S_{m - 1}}{\sqrt{m}} > -\epsilon, \max\limits_{i \leq m - 1} < \gamma_m } \p{X_m \geq \gamma_m + \epsilon \sqrt{m}}
 \\ & \geq \nonumber
 \intoo{\p{\frac{S_{m - 1}}{\sqrt{m}} > - \epsilon} - \p{\exists i \leq m -1 , X_i > \gamma_m}} m \p{X \geq \gamma_m + \epsilon \sqrt{m} }
\\ & \geq \label{proof:ld-poly:prob-lower-bound1}
 \intoo{\p{\frac{S_{m - 1}}{\sqrt{m}} > - \epsilon} -  (m - 1)\p{X > \gamma_m}} m \p{X \geq \gamma_m + \epsilon \sqrt{m} }.
 \end{align}
Note that by the central limit theorem we have $\p{\frac{S_{m - 1}}{\sqrt{m}} > - \epsilon} \geq \p{\frac{S_{m - 1}}{\sqrt{m}} > 0}\xrightarrow{m \to \infty} \frac{1}{2}$.  Furthermore,

\begin{equation} \label{proof:ld-poly:mp(X<gamma)}
(m - 1)\p{X > \gamma_m} = \expp{\log(m - 1) - I_{br}(\gamma_m)} \sim \expp{\log(m - 1) - \alpha \log (\gamma_m)}  \sim \expp{(1 - \frac{\alpha}{k}) \log m }.
\end{equation}
Since $k \leq 2 < \alpha$, the right hand side of \eqref{proof:ld-poly:mp(X<gamma)} goes to $0$ as $m$ grows. Hence for large enough $m$, we have 
\begin{equation*}
 \intoo{\p{\frac{S_{m - 1}}{\sqrt{m}} > - \epsilon} -(m-1) \p{X > \gamma_m}} \geq \frac{1}{3}.
\end{equation*}
Therefore, by \eqref{proof:ld-poly:prob-lower-bound1} we obtain
\begin{eqnarray*}
\lefteqn{\lim_{m \to \infty} \frac{- \log \p{S_m > \gamma_m}}{I(\gamma_m) - \log m}   \leq 
\lim_{m \to \infty}  \frac{\log 3 - \log \p{X > \gamma_m + \epsilon \sqrt{m}} - \log m}{\alpha \log \gamma_m - \log m}}\\ 
& = & \lim_{m \to \infty} \frac{\alpha \log \intoo{\gamma_m + \epsilon \sqrt{m}} - \log m}{\alpha \log \gamma_m - \log m}
 =  1 \ \ \ \ \ \ \ \ \ \ \ 
\end{eqnarray*}
To obtain the last equality we have used $\lim_{m \to \infty}  \frac{\log \intoo{\gamma_m + \epsilon \sqrt{m}}}{\log \gamma_m} = 1$ which can be easily proved by noting that $\log \gamma_m \leq \log \intoo{\gamma_m + \epsilon \sqrt{m}} \leq \log \gamma_m + \frac{\epsilon \sqrt{m}}{\gamma_m} $ and that $ \sqrt{m} \ll \gamma_m$ since $k < 2$.

\section{Conclusion} \label{sec:conclusion}

We developed a framework to study the concentration of the sum of independent and identically distributed random variables with heavy tails. In particular, we considered distributions for which the moment generating function does not exist.  Techniques that we offered in this paper are pretty simple and yet effective for all distributions that have finite variances. The generality and simplicity of the tools not only enable us to recognize different deviation behaviors, but also help us to determine the boundary of such phase transitions precisely.  Furthermore, we showed the tools that we developed for obtaining concentration inequalities are sharp enough to offer large deviation results as well. 
Note that there are plenty of results in the literature, such as Hanson-Wright inequality \cite{rudelson2013hanson} and Gartner-Ellis Theorem \cite{Dembo1998LargeDT}, whose proof  heavily relies on the moment generating function. We believe that the framework presented here can extend all such results to the class of distributions with finite variance.

\bibliographystyle{plain}
\bibliography{References.bib}


%


\end{document}